\documentclass[11pt,a4paper]{article}
\usepackage[a4paper, margin=3cm]{geometry}
\usepackage[T1]{fontenc}
\usepackage[utf8]{inputenc}
\usepackage{lmodern}
\usepackage{csquotes}

\ifx\french\undefined
\usepackage[english]{babel}
\else
\usepackage[french]{babel}
\fi

\usepackage{amsmath}
\usepackage{amsthm}
\usepackage{dsfont}
\usepackage{mathtools}
\usepackage{amssymb}
\usepackage{amsfonts}
\usepackage{mathrsfs}

\usepackage{graphicx}
\usepackage[shortlabels]{enumitem}
\usepackage[singlelinecheck=false,justification=centering]{caption}
\usepackage[hypcap]{subcaption} 

\usepackage[scaled=1,sups]{XCharter}
\usepackage[charter,varbb,smallerops,scaled=1.08]{newtxmath}
\usepackage{xcolor}
\usepackage{hyperref}

\definecolor{darkblue}{rgb}{0,0,0.5}
\definecolor{darkred}{rgb}{0.5,0,0}
\definecolor{darkgreen}{rgb}{0,0.5,0}

\hypersetup{
  unicode=true,          
  pdffitwindow=false,     
  pdfstartview={FitH},    
  pdftitle={},    
  pdfauthor={},     
  pdfsubject={},   
  pdfkeywords={}, 
  colorlinks=true,       
  linkcolor=darkblue,          
  citecolor=darkred,        
  filecolor=magenta,      
  urlcolor=darkgreen           
}

\usepackage[numbers,sort&compress]{natbib}
\usepackage{doi}

\usepackage{parskip}
\begingroup
\makeatletter
\@for\theoremstyle:=definition,remark,plain\do{%
  \expandafter\g@addto@macro\csname th@\theoremstyle\endcsname{%
    \addtolength\thm@preskip\parskip }%
}
\endgroup

\theoremstyle{plain}
\newtheorem{theorem}{Theorem}

\newtheorem{prop}[theorem]{Proposition}
\newtheorem{lemma}[theorem]{Lemma}

\theoremstyle{definition}
\newtheorem{definition}[theorem]{Definition}

\newtheorem{remark}[theorem]{Remark}

\newcommand{\Rplus}{[0,\infty)}
\newcommand{\R}{\mathbb{R}}
\newcommand{\N}{\mathbb{N}}

\renewcommand{\P}{\mathbb{P}}
\newcommand{\E}{\mathbb{E}}
\newcommand{\1}{\mathds{1}}
\newcommand{\ed}{\overset{(d)}{=}}

\newcommand{\dd}{\mathrm{d}}

\newcommand{\point}{\,\cdot\,}
\newcommand{\norm}[1]{\lVert #1 \rVert}
\newcommand{\abs}[1]{\lvert #1 \rvert}

\renewcommand{\epsilon}{\varepsilon}
\renewcommand{\rho}{\varrho}
\renewcommand{\phi}{\varphi}
\renewcommand{\emptyset}{\varnothing}

\newcommand{\partit}[1]{\mathcal{P}_{#1}}
\newcommand{\partitemb}[1]{\mathcal{P}_{#1}^{2,\preceq}}

\newcommand{\fraisse}{\mathscr{P}}
\newcommand{\monoline}{M}
\newcommand{\im}{\mathrm{im}}
\newcommand{\id}{\mathrm{id}}
\newcommand{\masspart}{\mathscr{P}_{\mathrm{m}}}
\newcommand{\masspartemb}{\mathscr{P}_{\mathrm{m},\preceq}}

\newcommand{\block}[1]{\hspace{.05\textwidth}\begin{minipage}{.90\textwidth} #1 \end{minipage}}

\newcommand{\restr}[1]{_{|[#1]}}

\title{\vspace{-2em} Trees within trees II: Nested Fragmentations}
\author{Jean-Jil \sc Duchamps \vspace{1em} \\ \sc Sorbonne Université \vspace{-1em}}
\date{\today}

\begin{document}

\maketitle

\begin{abstract}
  Similarly as in \cite{snec} where nested coalescent processes are studied, we generalize the definition of partition-valued homogeneous Markov fragmentation processes to the setting of nested partitions, i.e.\ pairs of partitions $(\zeta,\xi)$ where $\zeta$ is finer than $\xi$.
  As in the classical univariate setting, under exchangeability and branching assumptions, we characterize the jump measure of nested fragmentation processes, in terms of erosion coefficients and dislocation measures.
  Among the possible jumps of a nested fragmentation, three forms of erosion and two forms of dislocation are identified -- one of which being specific to the nested setting and relating to a bivariate paintbox process.
\end{abstract}

\begin{table}[b!]
  \small
  \rule{.5\linewidth}{.4pt}
  
  \smallskip
  
\end{table}

\vspace*{-1em}
{
\small
\tableofcontents
}

\section{Introduction}

Evolutionary biology aims at tracing back the history of species, by identifying and dating the relationships of ancestry between past lineages of extant individuals.
This information is usually represented by a tree or phylogeny, species
corresponding to leaves of the tree and speciation events (point in time where several species descend from a single one) corresponding to internal nodes~\cite{lambertProbabilistic2017,semplePhylogenetics2003}.

Modern methods consist in analyzing and comparing genetic data from samples of
individuals to statistically infer their phylogenetic tree.
Probabilistic tree models have been well-developed in the last decades -- either
from individual-based population models like the classical Wright-Fisher model \cite{Lam08,semplePhylogenetics2003,Ber09,Eth11},
or from time-forward branching processes, where the branching particles are species (see for instance Aldous's Markov branching models~\cite{aldousProbability1996} and the revolving literature \cite{chenNew2009,fordProbabilities2006,craneGeneralized2017,haasContinuum2008}) -- allowing for inference from genetic data.
A challenge is that trees inferred from different parts of the genome generally fail to coincide, each of them being understood as an alteration of a
``true'' underlying phylogeny (which we call the \emph{species tree}).

To understand the relation between \emph{gene trees} and the species tree, our goal is to identify a class of Markovian models coupling the evolution of both trees, making the assumption that in general, several gene lineages coexist within the same species, and at speciation events one or several gene lineages diverge from their neighbors to form a new species, i.e.\ we model the problem as a \emph{tree within a tree} \cite{pageTrees1998,pageGene1997,doyleTrees1997,maddisonGene1997}.
See Figure~\ref{fig:example} for an instance of a simple nested genealogy where discrepancies arise between the resulting gene tree and species tree.
\begin{figure}[ht]
\centering
\includegraphics[width=.7\linewidth]{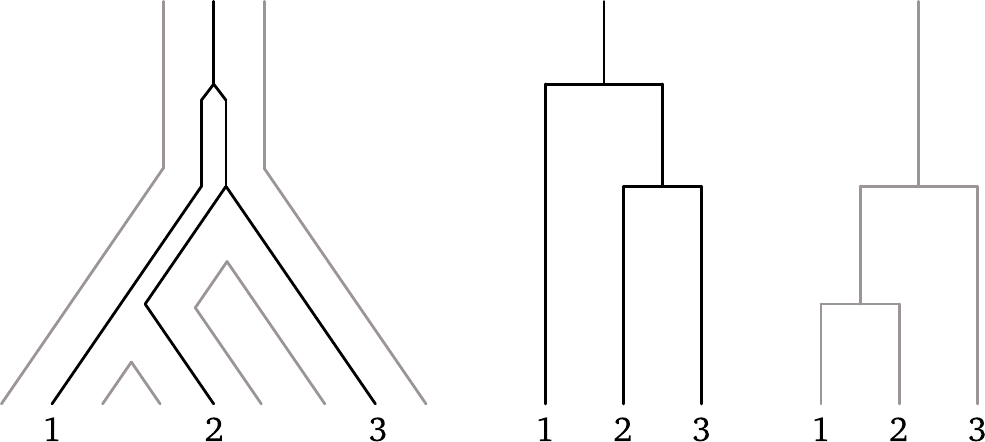}
\caption{Example of a nested tree where the gene tree (in black) does not coincide with the species tree (in gray).}\label{fig:example}
\end{figure}

Recent research aims at defining mathematical processes giving rise to such nested trees, generalizing several well-studied univariate (we will sometime use this term as opposed to ``nested'') processes.
Some work in progress involves a nested version~\cite{BRSS18,LS} of the Kingman coalescent~\cite{kingmanCoalescent1982} (considered the neutral model for evolution, appearing as a scaling limit of many individual-based population models).
In~\cite{snec} we study a nested generalization of $\Lambda$-coalescent processes \cite{sagitovGeneral1999,bertoinRandom2006,pitmanCoalescents1999} and characterize their distribution.
Our present goal is to generalize the forward-time branching models originated from Aldous~\cite{aldousProbability1996}.
His assumptions (which will be formally defined for our context in Section
\ref{sec:proj_markov}) are basically that the random process of evolution is
homogeneous in time and that the law of the process is invariant under both relabeling and resampling of individuals (we then say the process is \emph{exchangeable} and \emph{sampling consistent}).
We are interested in the partition-valued processes satisfying these assumptions,
i.e.\ the so-called fragmentation processes
\cite{bertoinRandom2006,haasContinuum2008},
and in this article we generalize their definition to \emph{nested
partition-valued} processes to model jointly a gene tree within a species tree.

Crane~\cite{craneGeneralized2017} also generalizes Aldous's Markov branching models to study the gene tree/species tree problem but uses a different approach to the one we use here.
Indeed, his model is such that first the entire species tree $\mathbf{t}$ is drawn according to some probability, and then the gene tree $\mathbf{t}'$ is constructed thanks to a generalized Markov branching model that depends on $\mathbf{t}$.
In the meantime, our goal is to characterize the class of models in which there is a joint Markov branching construction of both the gene tree and the species tree, under the assumptions of exchangeability and sampling consistency.

In particular our main result Theorem~\ref{thm:inner_charac}, which will be formally stated in Section~\ref{sec:inner_branch}, consists in showing that nested fragmentation processes satisfying natural branching properties are uniquely characterized by
\begin{itemize}
  \item three \emph{erosion parameters} $c_{\mathrm{out}},c_{\mathrm{in},1}$ and $c_{\mathrm{in},2}$ (rates at which a unique lineage can fragment out of its mother block, in three different situations);
  \item two \emph{dislocation measures} $\nu_{\mathrm{out}}$ and $\nu_{\mathrm{in}}$ that are Poissonian intensities of how blocks instantaneously fragment into several new blocks with macroscopic frequencies.
\end{itemize}

The article is organized as follows.
Section~\ref{sec:defs} briefly introduces some definitions and notation used throughout the paper.
In Section~\ref{sec:proj_markov} we define our exchangeability and sampling consistency properties -- or projective Markov property --, and show their equivalence to a ``strong exchangeability'' property in a fairly general setting.
We also recall some results in the univariate case which we seek to generalize to the nested case.
In Section~\ref{sec:outer_branch} we formulate some branching property assumptions, showing how they lead to simplifications in the representation of semi-groups of fragmentations, and giving a natural Poissonian construction of such processes.
Under an additional branching property assumption, Section \ref{sec:inner_branch} is devoted to the full characterization of the semi-group of simple nested fragmentation processes, in terms of \emph{erosion} and \emph{dislocation measures}.
It is shown that dislocations, similarly as in the univariate case, can be understood as (bivariate) paintbox processes.
Finally Section~\ref{sec:binary} briefly shows how our main result, Theorem~\ref{thm:inner_charac}, translates in simpler terms when we make the classical biological assumption that all splits are binary.

\section{Definitions, notation}\label{sec:defs}

For a set $S$, write $\partit S$ for the set of partitions of $S$:
\[ \partit{S} := \{\pi \subset \mathfrak{P}(S)\setminus\{\emptyset\},\; \forall A \neq B \in \pi, A\cap B = \emptyset \text{ and } \textstyle\bigcup_{A\in \pi} A = S \}, \]
where $\mathfrak{P}(S)$ denotes the power set of $S$.

For $S,S'$ two sets, $\pi\in \partit{S}$ and $\sigma : S' \to S$ an \textbf{injection}, we write
\[ \pi^{\sigma} := \{\sigma^{-1}(A), \; A \in \pi\}\setminus\{\emptyset\}, \]
and if $\mu$ is a measure on $\partit{S}$ then we write $\mu^\sigma$ for the push-forward of $\mu$ by the map $\pi \mapsto \pi^{\sigma}$.

Note that if $S'' \overset{\tau}{\rightarrow} S'\overset{\sigma}{\rightarrow} S$ are injections, then we have $\pi^{\sigma\tau} = (\pi^{\sigma})^{\tau}$, and $\mu^{\sigma\tau} = (\mu^{\sigma})^{\tau}$.

For $S' \subset S$, there is a natural surjective function $r_{S,S'} : \partit S \to \partit {S'}$ called the restriction, defined by
\[ r_{S,S'}(\pi) = \pi_{|S'} := \{A \cap S', \; A \in \pi\} \setminus \{\emptyset\}. \]
Note that $\pi_{|S'} = \pi^\sigma$ for $\sigma : S' \to S, x \mapsto x$ the canonical injection.

There is always a partial order on $\partit S$, denoted $\preceq$ and defined as:
\[ \pi \preceq \pi' \quad\text{ if }\quad \forall (A,B) \in \pi\times\pi', \, A \cap B \neq \emptyset \Rightarrow A \subset B, \]
that is $\pi \preceq \pi'$ if $\pi$ is finer than $\pi'$.
We will work on the space consisting of two nested partitions, which we will note $\partitemb{S}$:
\[ \partitemb{S} := \{(\zeta, \xi) \in \partit{S}^2, \; \zeta \preceq \xi\}. \]
We equip the space $\partitemb{S}$ with a partial order $\preceq$ defined naturally as
\[ (\zeta, \xi) \preceq (\zeta', \xi') \text{ if } \zeta \preceq \zeta' \text{ and } \xi \preceq \xi'. \]
%
Let us now define, for $n \in \N$, $[n] := \{1, \ldots, n\}$ and $[\infty] := \N$, and for $n \in \N \cup \{\infty\}$:
\[ \partit n := \partit{[n]} = \{ \zeta \text{ partition of } [n]\}. \]
We will generally label the blocks of a partition $\pi = \{\pi_1, \pi_2, \ldots\}$, in the unique way such that
\[ \min \pi_1 < \min \pi_2 < \ldots \]
The space $\partitemb{\infty}$ is endowed with a distance $d$ which makes it compact, defined as follows:
\[ d(\pi, \pi') = \left (\sup \{n \in \N, \; \pi\restr{n} = \pi\restr{n} \} \right )^{-1}, \]
with the convention $(\sup \N)^{-1} = 0$.

For $k\leq n \leq \infty$, $\sigma : [k] \to [n]$ an injection and $\pi = (\zeta, \xi) \in \partitemb{n}$, we write
\[ \pi^\sigma := (\zeta^\sigma, \xi^\sigma)\in\partitemb{k}. \]
Also, we write $\pi\restr{k} := (\zeta\restr{k}, \xi\restr{k}) \in \partitemb{k}$.

A measure $\mu$ on $\partit{n}$ or on $\partitemb{n}$ is said to be \textbf{exchangeable} if for any permutation $\sigma : [n] \to [n]$, we have
\[ \mu^{\sigma} = \mu. \]
A random variable $\Pi$ taking values in $\partit{n}$ or in $\partitemb{n}$ is said to be \textbf{exchangeable} if for any permutation $\sigma : [n] \to [n]$, we have
\[ \Pi^{\sigma} \ed \Pi, \]
that is if its distribution is exchangeable.
Similarly, a random process $(\Pi(t), t\geq 0)$ taking values in $\partit{n}$ or in $\partitemb{n}$ is said to be \textbf{exchangeable} if for any initial state $\pi_0$ and any permutation $\sigma : [n] \to [n]$, we have
\[ (\Pi(t)^{\sigma}, t\geq 0) \text{ under } \P_{\pi_0}\quad \ed \quad(\Pi(t), t\geq 0) \text{ under } \P_{\pi_0^{\sigma}}, \]
where $\P_{\pi}$ is the distribution of the process started from $\pi$.

Finally, a measure or a random process with values in $\partit{\infty}$ or $\partitemb{\infty}$ will be called \textbf{strongly exchangeable} if its distribution is invariant under the action of \emph{injections}.
Note that while for processes this is a strictly stronger assumption than being exchangeable (see Section~\ref{sec:strong_exch_markov_process}), for measures the two properties are equivalent.

In the following we only consider time-homogeneous Markov processes.

\section{Projective Markov property and strong exchangeability}\label{sec:proj_markov}

\subsection{Projective Markov process}

For each $n\in\N$, let $A_n$ be a finite non-empty set.
Assume there are surjective maps $r_{m,n}:A_m\to A_n$ for each $m\geq n$ which satisfy
\begin{align*}
\forall p\geq m\geq n\geq 1, &\quad r_{m,n} \circ r_{p,m} = r_{p,n}, \\
\forall n \in\N, &\quad r_{n,n} = \id_{A_n}.
\end{align*}
The family $(A_n, r_{m,n}, \, m\geq n\geq 1)$ is called a \textbf{finite inverse system}, and we can define the inverse limit
\[ A = \underleftarrow{\lim} \; A_n := \left \{(a_n, n\geq 1) \in \textstyle\prod_{n\in\N}A_n, \; \forall m \geq n, r_{m,n}(a_m) = a_n \right \}, \]
along with the canonical projection maps $r_n : A\to A_n, \; (a_n, n\geq 1) \mapsto a_n$.
A natural distance $d$ can be defined on the space $A$, by
\[ d(a,b) := \left (1/2+\sup\{n\geq 1, \; a_n = b_n\}\right )^{-1}, \]
where we use the conventions $\sup \emptyset = 0$ and $(1/2+\sup\N)^{-1} = 0$.
Note that its topology is then generated by the sets
\[ r_n^{-1}(\{a\}), \; n\geq 1, a\in A_n, \]
which are the balls of radius $1/n$ and center any $c\in r_n^{-1}(a)$.
The assumption that the sets $A_n$ are finite makes the space $(A,d)$ compact, so we can consider stochastic processes with values in $A$.

\begin{remark}
$\partit{\infty} = \underleftarrow{\lim}\; \partit{n}$ and $\partitemb{\infty} = \underleftarrow{\lim}\; \partitemb{n}$ are both inverse limits of finite inverse systems, where the restriction maps are $r_{m,n}: \partit{m} \to \partit{n}, \;\pi \mapsto \pi\restr{n}$.
\end{remark}

\begin{prop} \label{prop:projective_markov_kernel}
Let $X = (X(t), t\geq 0)$ be a stochastic process with values in $A$ the inverse limit of a finite inverse system.
Assume that the following \textbf{projective Markov property} holds:

\block{For all $n\geq 1$, the process $X^n := (r_n(X(t)), t\geq 0)$ is a continuous-time Markov chain in the finite state space $A_n$, whose distribution under $\P_{a}$ depends only on $r_n(a)$.}

Then $X$ is a Markov process, whose distribution is characterized by a transition kernel $K$ from $A$ to $A$ (i.e.\ $K_a(\point)$ is a nonnegative measure on $A$ for all $a\in A$ and $a\mapsto K_a(B)$ is measurable for any $B$ Borel set of $A$) such that
\begin{itemize}
  \item for all $a\in A$, we have $K_{a}(\{a\}) = 0$,
  \item for all $a\in A$ and $a'\in A_n\setminus\{r_n(a)\}$, the Markov chain $X^n$ has a transition rate from $r_n(a)$ to $a'$ equal to
  \[ q^n_{a, a'} = K_{a}\left (r_{n}^{-1}(\{a'\})\right ). \]
\end{itemize}
\end{prop}

\begin{proof}
$X^n$ is a Markov chain, therefore there exist transition rates 
\[ q^n_{a, a'} = \lim_{t\downarrow 0}\frac{1}{t} \P_{a}(X^n(t) = a') \]
for all $a\in A ,\, a' \in A_n\setminus\{r_n(a)\}$.
Now since for $n<m$, $X^m$ and $X^n = r_{m,n}(X^m)$ are both Markov chains, necessarily we have
\[ q^n_{a, a'} = \sum_{a'' \in r_{m,n}^{-1}(a')} q^m_{a, a''}. \]
Fix $a^\star\in A$ and $n\geq 1$ and consider the application
\[ f_n:a \in A_n \setminus \{r_n(a^\star)\} \longmapsto q^n_{r_n(a^\star), a}. \]
Then these applications $(f_n,\,n\geq 1)$ satisfy
\[ \forall m\geq n \geq 1,\, a\in A_n\setminus\{r_n(a^{\star})\}, \qquad  f_n(a) = \sum_{a' \in r_{m,n}^{-1}(\{a\})} f_m(a'). \]
It is then easy to check that Carathéodory's extension theorem allows us to build a measure $K_{a^{\star}}$ on $A\setminus \{a^\star\}$ (which we see as a measure on $A$ such that $K_{a^{\star}}(\{a^{\star}\}) = 0$) for which
\[ \forall n\geq 1,\, a\in A_n\setminus\{r_n(a^{\star})\}, \qquad K_{a^{\star}}\left (r_{n}^{-1}(\{a\})\right ) = f_n(a) = q^n_{r_n(a^{\star}), a}.\]

Let us check that $K$ is a kernel, i.e.\ that $a \mapsto K_{a}(B)$ is measurable for any Borel set $B$.
For $B$ of the form $r_n^{-1}(a')$, we have $K_{a}(B) = q^n_{r_n(a), a'}$, so $a \mapsto K_{a}(B)$ is clearly measurable.
It is readily checked that the sets $r_n^{-1}(a')$ form a $\pi$-system and that the sets $B$ such that $a \mapsto K_{a}(B)$ is measurable form a monotone class.
The monotone class theorem then implies that this property holds for any Borel set $B\subset A$.

Let us now show that $K$ characterizes uniquely the distribution of $X$.
Clearly, $K$ characterizes the distribution of $X^n$ for all $n\in\N$ since all the transition rates of the Markov chain $X^n$ can be recovered as a function of $K$.
By assumption, those distributions are consistent, in the sense that for any $m\geq n$, we have $r_{m,n}(X^m) \ed X^n$, where $\ed$ denotes equality in distribution.
Then, by Kolmogorov's extension theorem, there is a unique distribution for $X$ which satisfies $r_n(X) \ed X^n$ for all $n\in\N$.
\end{proof}

Let us now note $r_n(a) = a_n$ for any $a\in A$ to ease the notation.
Note that the infinitesimal generator $G_n$ of the continuous-time finite-space Markov chain $X^n$ is then given by
\begin{align*}
G_nf(a_n) &= \sum_{b_n\in A_n\setminus\{a_n\}} q^{n}_{a,b}(f(b_n)-f(a_n))\\
&= \int_{A} K_a(\dd b) \,\big(f(b_n)-f(a_n)\big),
\end{align*}
for any function $f:A_n\to\R$ and $a\in A$.
Let us see that this result holds in the limit $n\to\infty$, at least for a class of continuous functions $f:A\to\R$.
Whether the preceding result holds for a continuous function $f$ will depend on its modulus of continuity $\omega_f:\Rplus\to\Rplus$ defined for $\epsilon > 0$ by
\[ \omega_f(\epsilon) := \sup \{\abs{f(a)-f(a')}, \; a,a'\in A, d(a,a')\leq \epsilon\}, \]
which is always finite since $A$ is compact.

\begin{prop}
Let $X$ be a projective Markov process defined on the compact space $(A,d)$, inverse limit of a finite inverse system $(A_n, n \in\N)$, and consider its characteristic kernel $K$ as given by Proposition~\ref{prop:projective_markov_kernel}.

Let $k_n := \max_{a \in A} K_{a}(A\setminus r_n^{-1}(\{a_n\}))$ denote the maximum jump rate of the Markov chain $X^n$.
Consider a function $f:A\to\R$ with a modulus of continuity denoted by $\omega_f$, and suppose $\omega_f(1/n) k_{n+1}^2 \to 0$ as $n \to \infty$.

Then for every $a\in A$, the function $b\mapsto (f(b)-f(a))$ is $K_a$-integrable and the infinitesimal generator $G$ of the Markov process $X$ is well-defined on $f$ and satisfies
\begin{equation}\label{eq:inf_gen}
Gf(a) = \lim_{t\to 0}\frac{\E_{a} f(X_t) - f(a)}{t} = \int_{A}K_a(\dd b) \big(f(b)-f(a)\big).
\end{equation}
\end{prop}

\begin{proof}
First, note that if $k_n=0$ for all $n$, then $K_a=0$ for all $a\in A$ and 
the process $X$ is almost surely constant, so~\eqref{eq:inf_gen} is correct.
We now assume that $k_n > 0$ for $n$ large enough.

Fix $a\in A$. Let us first check that $b\mapsto (f(b)-f(a))$ is $K_a$-integrable.
Let $B_0:=A\setminus r_1^{-1}(\{a_n\})$ and for $n\geq 1$, $B_n:=r_{n}^{-1}(\{a_n\}) \setminus r_{n+1}^{-1}(\{a_{n+1}\})$, and notice that
\begin{align}
\int_{A}K_a(\dd b)\,\abs{f(b)-f(a)} &\leq K_a(B_0)\omega_f(2) + \sum_{n=1}^{\infty}\int_{B_n} K_a(\dd b)\,\omega_f(1/n) \nonumber \\
&= k_1 \omega_f(2) + \sum_{n=1}^{\infty}(k_{n+1}-k_n)\omega_f(1/n). \label{eq:sum-omega-k}
\end{align}
By assumption, $\omega_f(1/n)k_{n+1}^2 \to 0$, so we have $\omega_f(1/n) = o\big(k_{n+1}^{-2}\big)$, and since $(k_n)_n$ is a positive, nondecreasing sequence,
\[ \sum_{n=N}^{\infty}\frac{k_{n+1}-k_n}{k_{n+1}^2} \leq \sum_{n=N}^{\infty}\frac{k_{n+1}-k_n}{k_{n+1}k_n} = \sum_{n=N}^{\infty}\left (\frac{1}{k_n}-\frac{1}{k_{n+1}}\right ) \leq \frac{1}{k_N}, \]
which is finite for $N$ such that $k_N > 0$.
It follows that the sum in~\eqref{eq:sum-omega-k} is finite, so the function $b\mapsto (f(b)-f(a))$ is $K_a$-integrable.

Now for each $n\in\N$, consider a family $(a^{1},a^{2}, \ldots, a^{p})\in A^{p}$ such that $A_n = \{a_n, a^{1}_n, a^{2}_n, \ldots, a^{p}_n\}$ with no repetition, i.e.\ such that $p+1=\abs{A_n}$.
Now let us define for all $b \in A $, $f_n(b) := f(a^{i})$ if and only if $b_n=a^{i}_n$.
Notice that $f_n$ is an approximation of $f$, in the sense that the error function $g_n:b\mapsto (f(b)-f_n(b))$ necessarily satisfies $\abs{g_n(b)} \leq \omega_f(1/n)$.
Note also that by definition, $f_n(a)=f(a)$.

Let us here treat the case when there exists $n\geq 1$ such that $\omega_f(1/n) = 0$.
By the preceding remark, we have $f_n=f$, in other words there exists an application $\widetilde{f}_n:A_n \to \R$ such that $f(b)=\widetilde{f}_n(b_n)=\widetilde{f}_n(r_n(b))$.
So $\E_a f(X_t)=\E_a\widetilde{f}_n(r_n(X_t))$, and since $(r_n(X_t), t\geq 0)$ is a finite-state-space continuous-time Markov chain, it is immediate that
\[ \E_{a}f(X_{t}) = f(a) + t\Big (\sum_{i=1}^{p} q^n_{a,a^i} (f(a^i)-f(a)) \Big) + O\big((t k_n)^2 \norm{f}_\infty\big), \]
where $\norm{f}_{\infty} := \sup_{b\in A}\abs{f(b)}$, and where the constant in the term $O\big((t k_n)^2 \norm{f}_\infty\big)$ does not depend on $t$, $K$ or $f$.
From this it is clear that 
\[  \frac{\E_{a} f(X_{t}) - f(a)}{t} \underset{t\to 0}{\longrightarrow} \sum_{i=1}^{p} q^n_{a,a^i} (f(a^i)-f(a)) = \int_{A} K_a(\dd b)(f(b)-f(a)). \]

Now let us assume that for all $n\geq 1$, $\omega_f(1/n) > 0$.
Since $f_n(b)$ depends only on $b_n$, we can write
\begin{align*}
\E_{a}f_n(X_{t}) &= f(a) + t\int_{A}K_a(\dd b)(f_n(b)-f(a)) + O\big((t k_n)^2 \norm{f}_\infty\big)\\
&= f(a) + t\int_{A\setminus r_n^{-1}(\{a_n\})}\!\!\!K_a(\dd b)(f(b)-f(a)) + O(t\omega_f(1/n) k_n)+O\big((t k_n)^2 \norm{f}_\infty\big),
\end{align*}
Notice also that 
\[ \left \lvert\frac{\E_{a} f(X_{t}) - \E_{a} f_n(X_{t})}{t}\right \rvert \leq \frac{\omega_f(1/n)}{t}, \]
so that putting everything together, we have
\begin{equation}\label{eq:final_inf_gen}
 \frac{\E_{a} f(X_{t}) - f(a)}{t} = \int_{A\setminus r_n^{-1}(\{a_n\})}\!\!\!K_a(\dd b)\,(f(b)-f(a)) + O\left (\omega_f(1/n)k_n + \frac{\omega_f(1/n)}{t} + tk_n^2\right ). 
\end{equation}
If one can find $n=n(t)$ such that $n\to\infty$, $\omega_f(1/n)/t \to 0$ and $tk_n^2 \to 0$ as $t\to 0$, then passing to the limit in~\eqref{eq:final_inf_gen}, by using the dominated convergence theorem for the integral, yields~\eqref{eq:inf_gen}.

Now let us define for all $m\geq 1$, $t_m := \sqrt{\omega_f(1/m)}/k_p$ and $t'_m := \sqrt{\omega_f(1/m)}/k_{m+1}$.
Notice that
\[ t_{m} \geq t'_m \geq t_{m+1} \underset{m\to\infty}{\longrightarrow} 0, \]
so for each $t\in (0,t_1]$, there is an $m\geq 1$ such that $t \in [t_{m+1}, t_m]$.
Then,
\begin{itemize}
\item if $t\geq t'_m$, let $n(t):=m$, and we check
\[ \omega_f(1/n)/t \leq \omega_f(1/n)/t'_n = \sqrt{\omega_f(1/n)}k_{n+1}, \quad \text{and}\quad tk_n^2 \leq t_nk_n^2 =  \sqrt{\omega_f(1/n)}k_{n} ; \]
\item if $t\leq t'_m$, let $n(t):=m+1$, and we check
\[ \omega_f(1/n)/t \leq \omega_f(1/n)/t_{n} = \sqrt{\omega_f(1/n)}k_{n}, \quad \text{and}\quad tk_n^2 \leq t'_{n-1}k_n^2 =  \sqrt{\omega_f(1/(n-1))}k_{n}. \]
\end{itemize}
Since we assumed that $\omega_f(1/n) > 0$ for all $n$, then $t_m > 0$ for all $m$, which implies that necessarily $n(t) \to\infty$ as $t\to 0$.
Finally, the assumption that $\omega_f(1/n)k_{n+1}^2 \to 0$ as $n\to\infty$ ensures us that both $\omega_f(1/n)/t$ and $tk_n^2$ tend to $0$ as $t\to 0$, which concludes the proof.
%
\end{proof}


We are now interested in exchangeable projective Markov processes with values in the space of nested partitions $\partitemb{\infty}$, as an extension of univariate fragmentation processes (with values in $\partit{\infty}$).

\subsection{Strongly exchangeable Markov process} \label{sec:strong_exch_markov_process}

In the following, we write $\fraisse$ for either $\partit{\infty}$ or $\partitemb{\infty}$, when our assertions are valid for both spaces.
We will also write $\fraisse_{n}$ for $\partit{n}$ or $\partitemb{n}$.
A key property of those spaces is the following.

\block{For any $n\in \N$, and any $\pi\in\fraisse_n$, there is a $\pi^{\star}\in \fraisse$ satisfying:
\begin{itemize}
\item $\pi^{\star}\restr{n} = \pi$
\item for any $\pi' \in\fraisse$ such that $\pi'\restr{n} = \pi$, there is an injection $\sigma : \N \to \N$ which satisfies $\sigma\restr{n} = \id_{[n]}$ and $(\pi^{\star})^{\sigma} = \pi'$.
\end{itemize}
}

Indeed for instance in $\fraisse = \partit{\infty}$, it is easy to choose a $\pi^{\star}$ with an infinity of infinite blocks and no finite blocks, and such that $\pi^{\star}\restr{n} = \pi$.
This partition satisfies immediately the required property.
We will call any such $\pi^{\star}$ a \textbf{universal element of $\fraisse$ with initial part $\pi$} whenever we need to use one.

\begin{prop}
  Let $\Pi=(\Pi(t), t\geq 0)$ be an exchangeable Markov process taking values in $\fraisse$ with càdlàg sample paths.
  The following propositions are equivalent:
  \begin{enumerate}[(i)]
    \item $\Pi$ is strongly exchangeable.
    \item $\Pi$ has the projective Markov property, i.e.\ $\Pi^n := (\Pi(t)\restr{n}, t\geq 0)$ is a Markov chain for all $n\in \N$.
  \end{enumerate}
\end{prop}

\begin{remark}
  \citeauthor{craneStructure2016}~\cite[Theorem 4.26]{craneStructure2016} show that the projective Markov property is equivalent to the Feller property for exchangeable Markov process taking values in a Fraïssé space (i.e.\ a space satisfying general ``stability and universality'' assumptions \cite[see][Definitions 4.4 to 4.11]{craneStructure2016}).
  In particular the space of partitions and the space of nested partitions 
  are Fraïssé spaces (the argument essentially being the existence of so-called universal elements $\pi^{\star}$), so for the processes we consider, strong exchangeability is equivalent to the Feller property.
\end{remark}

\begin{proof}
  $(i) \Rightarrow (ii)$:
  Let $n \in \N$ and $\pi\in\fraisse_{n}$.
  Fix a universal $\pi^{\star}\in\fraisse$ with initial part $\pi$.
  Now take any $\pi_0\in\fraisse$ such that $(\pi_{0})\restr{n} = \pi$, and an injection $\sigma:\N\to\N$ such that $\sigma\restr{n} = \id\restr{n}$ and $(\pi^\star)^\sigma = \pi_0$.
  Now we have
  \begin{align*}
    \P_{\pi_0}( \Pi^n \in \cdot) &= \P_{\pi^{\star}}( (\Pi^{\sigma})^{n} \in \cdot)\\
    &= \P_{\pi^{\star}}( \Pi^{n} \in \cdot),
  \end{align*}
  so this distribution depends only on $\pi$, which proves that $\Pi^n$ is a Markov process.
  Now the assumption that $\Pi$ has càdlàg sample paths ensures that the process $\Pi^n$ stays some positive time in each visited state \textit{a.s.}
  Therefore $\Pi^n$ is a continuous-time Markov chain.
  
  $(ii) \Rightarrow (i)$:
  Let $\sigma:\N\to\N$ be an injection.
  For $n\in \N$, let $\tau$ be a permutation of $\N$ such that $\tau\restr{n} = \sigma\restr{n}$.
  This property implies $(\pi^{\tau})\restr{n} = (\pi^{\sigma})\restr{n}$ for any $\pi \in\fraisse$.
  We deduce
  \begin{align*}
  \P_{\pi}( (\Pi^{\sigma})^n \in \cdot) &= \P_{\pi}( (\Pi^{\tau})^{n} \in \cdot)\\
  &= \P_{\pi^{\tau}}( \Pi^{n} \in \cdot)\\
  &= \P_{\pi^{\sigma}}( \Pi^{n} \in \cdot)
  \end{align*}
  where the last equality is a consequence of the projective Markov property (the distribution of $\Pi^{n}$ under $\P_{\pi}$ depends only on the initial segment $\pi\restr{n}$).
  Since it is true for all $n$, we have $\P_{\pi}(\Pi^{\sigma} \in \cdot) = \P_{\pi^{\sigma}}(\Pi \in \cdot)$, which proves the property of strong exchangeability.
\end{proof}

\begin{remark}
  To be strongly exchangeable is strictly stronger than being exchangeable. 
  To see that, define the Markov process $\Pi = (\Pi(t), t\geq 0)$ taking values in $\partit{\infty}$ by:
  \begin{itemize}
    \item If $\pi\in \partit \infty$ has an infinite number of blocks, then let $\Pi$ under $\P_{\pi}$ be almost surely the constant function equal to $\pi$.
    \item If $\pi\in \partit \infty$ has a finite number of blocks, let $T$ be an Exponential($1$) random variable, and let the distribution of $\Pi$ under $\P_{\pi}$ be that of the random function:
    \[ t \mapsto
    \begin{cases}
      \pi & \text{ if } t < T\\
      \mathbf{0}_{\infty} & \text{ if } t \geq T
    \end{cases} \]
  \end{itemize}
  Then $\Pi$ is clearly exchangeable but not strongly exchangeable.
\end{remark}

\begin{prop}
  Let $\Pi = (\Pi(t), t\geq 0)$ be a strongly exchangeable Markov process in $\fraisse$.
  Then there is a unique kernel $K$ from $\fraisse$ to $\fraisse$ such that
  \begin{itemize}
    \item for all $\pi_0\in\fraisse$, we have $K_{\pi_0}(\{\pi_0\}) = 0$,
    \item for all $\pi_1\in\fraisse_n$, for all $\pi_2\in\fraisse_n\setminus\{\pi_1\}$, the Markov chain $\Pi^n$ has a transition rate from $\pi_1$ to $\pi_2$ equal to
    \[ K_{\pi_0}\left (\pi\restr{n}= \pi_2 \right ), \]
    where $\pi_0$ is any element of $\fraisse$ such that $(\pi_0)\restr{n} = \pi_1$.
  \end{itemize}
Furthermore this kernel is strongly exchangeable, i.e.\ for any $\pi_0\in\fraisse$ and any injection $\sigma : \N \to \N$, we have
\[ K_{\pi_0}^\sigma = K_{\pi_0^{\sigma}}. \]
\end{prop}

\begin{proof}
The first part of the proposition is an immediate consequence of Proposition~\ref{prop:projective_markov_kernel}.
It remains only to prove that $K$ is strongly exchangeable.
Consider $\pi_0\in\fraisse,\,n\in\N$, $\pi'\in\fraisse_n\setminus\{(\pi_0)\restr{n}\}$ and an injection $\sigma:\N\to\N$.
We have
\[ \frac{1}{t}\, \P_{\pi_0}\left ((\Pi(t)^{\sigma})\restr{n} = \pi'\right ) = \frac{1}{t}\, \P_{\pi_0^{\sigma}}\left (\Pi(t)\restr{n} = \pi'\right ) \]
because of the exchangeability of $\Pi$, and taking limits we find
\[ K_{\pi_0}\left ((\pi^{\sigma})\restr{n} = \pi'\right ) = K_{\pi_0^{\sigma}}\left (\pi\restr{n} = \pi'\right ). \]
So the two $\sigma$-finite measures $K_{\pi_0}^{\sigma}$ and $K_{\pi_0^{\sigma}}$ coincide on the sets of the form $\{\pi\restr{n} = \pi'\}$, which constitute a $\pi$-system generating the Borel sets of $\fraisse$.
Therefore they are equal, which concludes the proof.
\end{proof}

\begin{remark}
  Consider a universal element $\pi^{\star}\in\fraisse$ such that for any $\pi\in\fraisse$, there is an injection $\sigma$ such that $\pi = (\pi^{\star})^{\sigma}$.
  The exchangeability property of the kernel $K$ then implies that $K_{\pi} = K_{\pi^{\star}}^{\sigma}$, therefore $K$ is entirely determined by the single measure $K_{\pi^{\star}}$.
\end{remark}

\subsection{Univariate results, mass partitions} \label{sec:univariate_results}

Random exchangeable partitions $\pi \in\partit{\infty}$ and their relation to random mass partitions is well known \cite[see][Chapter 2]{bertoinRandom2006}.
Let us recall briefly some definitions and results, which we will then extend to the nested case.
We define the space of mass partitions
\begin{equation}\label{eq:masspart}
 \masspart := \left \{\mathbf{s} = (s_1, s_2, \ldots) \in [0,1]^{\N},\; s_1\geq s_2\geq \ldots, \; \textstyle\sum_{k} s_k \leq 1\right \}. 
\end{equation}
For $\mathbf{s} \in\masspart$, one defines an exchangeable distribution $\rho_{\mathbf{s}}$ on $\partit{\infty}$, by the following so-called \emph{paintbox construction}:
\begin{itemize}
\item for $k\geq 0$, define $t_k = \sum_{k'=1}^{k} s_{k'}$, with $t_0=0$ by convention.
\item let $(U_i, i\geq 1)$ be an i.i.d.\ sequence of uniform random variables in $[0,1]$.
\item define the random partition $\pi\in\partit{\infty}$ by setting
\[i\sim^{\pi}j \iff i=j \text{ or } \exists k\geq 1,\, U_i,U_j\in[t_{k-1}, t_k).\]
\end{itemize}
Note that the set $\pi_0 := \{[t_{k-1}, t_k), \, k\geq 1\}\cup\{\{t\}, \; \sum_{k\geq 1} s_k \leq t \leq 1\}$ is a partition of $[0,1]$, and that we have $\pi = \pi_0^{\sigma}$, where $\sigma:\N\to[0,1]$ is the random injection defined by $\sigma :i\mapsto U_i$.
Also, note that by definition some blocks are singletons (blocks $\{i\}$ such that $U_i\in[\sum_{k\geq 1}s_k,1]$), and by construction we have
\[ \frac{\#\{i\in[n],\,\{i\}\in\pi\}}{n}\underset{n\to\infty}{\longrightarrow} s_0 := 1-\textstyle\sum_{k\geq 1}s_k. \]
These integers that are singleton blocks are called the \emph{dust} of the random partition $\pi$ and the last display tells us there is a frequency $s_0$ of dust.

Conversely, any random exchangeable partition $\pi$ has a distribution that can be expressed with these paintbox constructions $\rho_{\mathbf{s}}$.
Indeed, $\pi$ has \textbf{asymptotic frequencies}, i.e.\
\[ \abs{B} := \lim_{n\to\infty} \frac{\#(B\cap[n])}{n} \quad \text{exists a.s. for all }B\in\pi. \]
Let us write $\abs{\pi}^{\downarrow}\in\masspart$ for the decreasing reordering of $(\abs{B}, B\in\pi)$, ignoring the zero terms coming from the dust.
Now it is known \cite[Theorem 2]{kingmanCoalescent1982} that the conditional distribution of $\pi$ given $\abs{\pi}^{\downarrow} = \mathbf{s}$ is $\rho_{\mathbf{s}}$, so we have
\[ \P(\pi\in\point) = \int\P(\abs{\pi}^{\downarrow} \in \dd \mathbf{s}) \rho_{\mathbf{s}}(\point). \]
This means that any exchangeable probability measure on $\partit{\infty}$ is of the form $\rho_{\nu}$ where $\nu$ is a probability measure on $\masspart$, and 
\[ \rho_{\nu}(\point) := \int_{\masspart}\rho_{\mathbf{s}}(\point) \nu(\dd \mathbf{s}). \]
Furthermore, Bertoin~\cite[Theorem 3.1]{bertoinRandom2006} shows that any exchangeable measure $\mu$ on $\partit{\infty}$ such that
\begin{equation}\label{eq:mu_sig_finite_simple}
\forall n\geq 1, \quad \mu(\pi\restr{n} \neq \mathbf{1}_{[n]}) < \infty
\end{equation}
can be written $\mu = c\mathfrak{e} + \rho_{\nu}$, where $c\geq 0$, $\nu$ is a measure on $\masspart$ satisfying
\begin{equation}\label{eq:nu_condition_simple}
\int_{\masspart}(1-s_1) \nu(\dd\mathbf{s}) < \infty,
\end{equation}
and $\mathfrak{e}$ is the so-called \emph{erosion measure}, defined by
\[ \mathfrak{e} := \textstyle \sum_{i\in\N} \delta_{\{\{i\},\N\setminus\{i\}\}}. \]

As a result, each fragmentation process with values in $\partit{\infty}$ is characterized by its erosion coefficient $c$ and characteristic measure $\nu$, in such a way that its rates can be described as follows:

\block{A block of size $n$ fragments, independently of the other blocks, into a partition with $k$ different blocks of sizes $n_1,n_2,\ldots,n_k$ with rate
\[ c\1\{k=2,\text{ and }\,n_1=1\text{ or }n_2=1\}+\int_{\masspart}\nu(\dd \mathbf{s})\sum_{\mathbf{i}}s^{n_1}_{i_1}\cdot s^{n_2}_{i_2}\cdots s^{n_k}_{i_k}, \]
where $s_0$ is defined to be $1-\sum_{i\geq 1} s_i$, and the sum is over the vectors $\mathbf{i}=(i_1,\ldots,i_k)\in\{0,1,\ldots\}^{k}$ such that $i_j$ may be $0$ only if $n_j=1$, and if $j\neq j'$ and $i_j\neq 0$, then $i_{j'}\neq i_{j}$.}

We aim at showing a similar result concerning fragmentations of nested partitions.

\section{Outer branching property} \label{sec:outer_branch}

From now on, to be able to give a more precise characterization of nested fragmentation processes, we will exclude from the study those processes which exhibit simultaneous fragmentations in separate blocks.
That is, we will assume a branching property: two different blocks at a given time undergo two independent fragmentations in the future.
In the univariate case, Bertoin~\cite[Definition 3.2]{bertoinRandom2006} expresses the branching property thanks to the introduction of a mapping $\text{Frag}:\partit{\infty}\times\partit{\infty}^{\N}\to\partit{\infty}$.
While a similar definition could be made in the nested case, the analog of the $\text{Frag}$ mapping would be too lengthy to introduce and we found simpler to assume an equivalent fact, which is all we will use in later proofs: distinct blocks fragment at distinct times.

We also need to distinguish two branching properties in the case of nested fragmentations, each concerning either the outer or the inner blocks (branching property for $\xi$ or for $\zeta$).

\begin{definition}
Let $\Pi = (\Pi(t), t\geq 0) = ((\zeta(t), \xi(t)), t\geq 0)$ be a strongly exchangeable Markov process with values in $\partitemb{\infty}$ and decreasing càdlàg sample paths.
We say that $\Pi$ satisfies the \textbf{outer branching property} if
  
  \block{Almost surely for all $t$ such that $\Pi(t-) \neq \Pi(t)$, there is a unique block $B \in \xi(t-)$ such that $\Pi(t-)_{|B}\neq\Pi(t)_{|B}$.}

Moreover, we say that $\Pi$ satisfies the \textbf{inner branching property} if

  \block{Almost surely for all $t$ such that $\zeta(t-) \neq \zeta(t)$, there is a unique block $B \in \zeta(t-)$ such that $\zeta(t-)_{|B} \neq \zeta(t)_{|B}$.}

Nested fragmentations processes satisfying both branching properties will be called \textbf{simple}.
\end{definition}

The rest of the paper is dedicated to characterize as simply and precisely as possible simple nested fragmentations processes.


\begin{prop} \label{prop:outer_branching}
Let $\Pi = (\Pi(t), t\geq 0) = ((\zeta(t), \xi(t)), t\geq 0)$ be a strongly exchangeable Markov process with values in $\partitemb{\infty}$ and decreasing càdlàg sample paths.
Write $K$ for its exchangeable characteristic kernel.
  
If $\Pi$ satisfies the \textbf{outer branching property}, then the characteristic kernel $K$ is characterized by a simpler kernel $\kappa$ from $\partit{\infty}$ to $\partitemb{\infty}$ which is defined as
\[ \kappa_{\zeta}(\point) := K_{(\zeta, \mathbf{1})}(\point), \]
where $\mathbf{1}$ denotes the partition of $\N$ with only one block.
The simpler kernel is also strongly exchangeable.

The kernel $K$ is determined by $\kappa$ in the following way: fix $\pi_0 = (\zeta, \xi)\in\partitemb{\infty}$ and for simplicity suppose that all the blocks of $\xi$ are infinite.
For all $B\in\xi$, define an injection $\sigma_{B} : \N \to \N$ whose image is $B$, and $\tau_B : B \to \N$ such that $\sigma_{B} \circ \tau_{B} = \id_{B}$.
By definition, $(\pi_0)^{\sigma_B}$ is of the form $(\zeta_B, \mathbf{1})$, with $\zeta_B = \zeta^{\sigma_{B}}$.
Now define $f_B$ as the application which maps $\pi\in\partitemb{\infty}$ to the unique $\omega \in\partitemb{\infty}$ such that
\begin{itemize}
\item $\omega \preceq (\{B,\N\setminus B\},\{B,\N\setminus B\})$,
\item $\omega_{|B} = \pi^{\tau_B}$ and $\omega_{|\N\setminus B} = (\pi_0)_{|\N\setminus B}$.
\end{itemize}
Then for any Borel set $A\subset\partitemb{\infty}$, we have
\begin{equation*}
K_{\pi_0}(A) = \sum_{B\in\xi}\kappa_{\zeta_B}(\{f_B(\pi)\in A\}\cap\{\pi \neq (\pi_0)^{\sigma_{B}}\}).
\end{equation*}
\end{prop}

\begin{remark} \label{rq:nolog}
This proposition shows how $K_{\pi_0}$ is expressed in terms of the kernel $\kappa$ only for $\pi_0=(\zeta,\xi)$ such that all the blocks of $\xi$ are infinite.
In fact this is enough to characterize $K$ entirely since if $\pi_0$ does not satisfy this property, there exists a nested partition $\pi_0'=(\zeta',\xi')$ which does and an injection $\sigma : \N \to \N$ such that $\pi_0 = (\pi_0')^{\sigma}$.
Then we have $K_{\pi_0} = K_{\pi_0'}^{\sigma}$, which is determined by $\kappa$.
\end{remark}

\begin{proof}
First note  that the fact that $\Pi$ has decreasing sample paths implies that for any $\pi_0\in \partitemb{\infty}$, the support of the measure $K_{\pi_0}$ is included in $\{\pi \preceq \pi_0\}$.
Indeed, since $\{\pi \preceq \pi_0\} = \cap_{n\geq 1}\{\pi\restr{n} \preceq (\pi_0)\restr{n}\}$, we have
\[ K_{\pi_0}(\{\pi \npreceq \pi_0\}) = \lim_{n\to\infty} K_{\pi_0}(\pi\restr{n} \npreceq (\pi_0)\restr{n}), \]
where for any $n\geq 1$, the right-hand side is equal to the (finite) transition rate of the Markov chain $\Pi^n$ from $(\pi_0)\restr{n}$ to any $\pi$ for which $\pi \npreceq (\pi_0)\restr{n}$.
But $\Pi^n$ is a decreasing process by assumption, so this rate is zero, so we conclude
\begin{equation}\label{eq:kappa_decreasing}
K_{\pi_0}(\pi \npreceq \pi_0) = 0
\end{equation}

Using the same argument, it is clear that the outer branching property implies that for any $\pi_0=(\zeta, \xi)\in\partitemb{\infty}$, we have
\begin{equation} \label{eq:outer_branching_kappa}
K_{\pi_0}\bigg ( \bigcup_{B_1\neq B_2 \in\xi} \{ \pi_{|B_1} \neq (\pi_0)_{|B_1} \text{ and } \pi_{|B_2} \neq (\pi_0)_{|B_2} \} \bigg) = 0.
\end{equation}

Now without loss of generality (see Remark~\ref{rq:nolog}), suppose that all the blocks of $\xi$ are infinite, and let us define for all $B\in\xi$, an injection $\sigma_{B} : \N \to \N$ whose image is $B$, and $\tau_B : B \to \N$ such that $\sigma_{B} \circ \tau_{B} = \id_{B}$.
Equations~\eqref{eq:kappa_decreasing} and~\eqref{eq:outer_branching_kappa} imply that for any $B\in \xi$, on the event $\{\pi_{|B} \neq (\pi_0)_{|B}\}$, we have
\[ \pi = f_B(\pi^{\sigma_B}) \quad K_{\pi_0}\text{-a.e.}, \]
where $f_B$ is the application defined in the proposition.
Then for any Borel set $A\subset\partitemb{\infty}$, we have
\begin{align*}
K_{\pi_0}(A) &= K_{\pi_0}(\cup_{B\in\xi}(A\cap\{\pi_{|B} \neq (\pi_0)_{|B}\}))\\
&= \sum_{B\in\xi}K_{\pi_0}(A\cap\{\pi_{|B} \neq (\pi_0)_{|B}\})\\
&= \sum_{B\in\xi}K_{\pi_0}(\{f_B(\pi^{\sigma_B})\in A\}\cap\{\pi^{\sigma_B} \neq (\pi_0)^{\sigma_{B}}\})\\
&= \sum_{B\in\xi}K_{(\pi_0)^{\sigma_B}}(\{f_B(\pi)\in A\}\cap\{\pi \neq (\pi_0)^{\sigma_{B}}\}).
\end{align*}
Now by definition of $\sigma_B$, $(\pi_0)^{\sigma_B}$ is of the form $(\zeta_B, \mathbf{1})$, which concludes the proof that $K_{\pi_0}$ can be expressed with the simpler kernel $\kappa$.
Finally, by definition, it is clear that $\kappa$ inherits the strong exchangeability from $K$.
\end{proof}

Now, to further analyze the ``simplified characteristic kernel'' $\kappa$ of an outer branching fragmentation, we need to introduce some tools, reducing the problem to study exchangeable (with respect to a particular set of injection $\monoline$) partitions on $\N^{2}$.

\subsection{\texorpdfstring{$\monoline$}{M}-invariant measures}

Let $\monoline$ be the monoid of applications $\N^2\to\N^2$ consisting of injective maps of the form
\[ (i,j) \longmapsto (\sigma(i), \sigma_i(j)), \]
where $\sigma$ and $\sigma_1, \sigma_2, \ldots$ are injections $\N\to\N$.
Let us write $\pi_{\mathrm{R}}$ for the ``rows partition'' $\{\{(i,j),\,j\geq 1\},\; i\geq 1\} \in\partit{\N^2}$, which is the minimal non-trivial (i.e.\ different from $\mathbf{0}$) $\monoline$-invariant partition.

\begin{prop} \label{prop:kappa_to_mu}
Let $\kappa$ be a strongly exchangeable kernel from $\partit{\infty}$ to $\partitemb{\infty}$, and let $\pi_0$ denote a partition of $\N$ with an infinity of infinite blocks (and no finite block).
Choose a bijection $\sigma: \N^2 \rightarrow \N$ such that $\pi_0^{\sigma} = \pi_{\mathrm{R}}$.

Then $\mu := \kappa_{\pi_0}^{\sigma}$ is a measure on $\partitemb{\N^2}$ which is $\monoline$-invariant.
Moreover, $\mu$ does not depend on $\pi_0$ or $\sigma$ and the mapping $\kappa \mapsto \mu$ is bijective from the set of strongly exchangeable kernels to the set of $\monoline$-invariant measures on $\partitemb{\N^2}$.
\end{prop}

\begin{proof}
Fix $\tau \in\monoline$ and a Borel set $A \subset \partitemb{\N^2}$.
We need to prove $\mu(\pi^\tau \in A) = \mu(A)$.
Consider $\phi = \sigma \circ \tau \circ \sigma^{-1}$.
This application satisfies $\phi\circ\sigma = \sigma\circ\tau$ and $\pi_0^{\phi} = \pi_0$, so we have
\begin{align*}
\mu(\pi^\tau \in A) &= \kappa_{\pi_0}(\pi^{\sigma\circ\tau}\in A)\\
&= \kappa_{\pi_0}(\pi^{\phi\circ\sigma}\in A)\\
&= \kappa_{\pi_0^{\phi}}(\pi^{\sigma}\in A)\\
&= \mu(A).
\end{align*}
This proves that $\mu$ is $\monoline$-invariant.
Let us now prove that $\mu$ does not depend on $\pi_0$ or $\sigma$: fix $\pi_1, \pi_2 \in\partit{\infty}$ (both with an infinity of infinite blocks and no finite block) and $\sigma_1,\sigma_2$ bijections from $\N^2 $ to $\N$ such that $\pi_i^{\sigma_i} = \pi_{\mathrm{R}}$.
We need to show 
\[\kappa_{\pi_1}(\pi^{\sigma_1}\in\point) = \kappa_{\pi_2}(\pi^{\sigma_2}\in\point).\]
Let $\phi$ be a bijection such that $\pi_1^\phi = \pi_2$.
Note that $\pi_{\mathrm{R}}^{\sigma_2^{-1}\circ\phi^{-1}\circ\sigma_1} = \pi_2^{\phi^{-1}\circ\sigma_1} = \pi_1^{\sigma_1} = \pi_{\mathrm{R}}$, i.e.\ $\sigma_2^{-1}\circ\phi^{-1}\circ\sigma_1 \in\monoline$.
Now we have
\begin{align*}
\kappa_{\pi_1}(\pi^{\sigma_1} \in \point) &= \kappa_{\pi_1}\left ((\pi^{\phi})^{\phi^{-1}\circ\sigma_1}\in\point\right )\\
&= \kappa_{\pi_2}\left (\pi^{\phi^{-1}\circ\sigma_1}\in\point\right )\\
&= \kappa_{\pi_2}\left ((\pi^{\sigma_2})^{\sigma_2^{-1}\circ\phi^{-1}\circ\sigma_1}\in\point\right )\\
&= \kappa_{\pi_2}\left (\pi^{\sigma_2}\in\point\right ),
\end{align*}
where the last equality follows from the $\monoline$-invariance of $\kappa_{\pi_2}\left (\pi^{\sigma_2}\in\point\right )$.
So $\mu$ is well defined and depends only on $\kappa$.

We now prove that $\kappa \mapsto\mu$ is bijective.
For any injection $\sigma : \N\to\N^2$, we write $2\sigma$ for the application
\[ 2\sigma : \begin{cases}
\N &\longrightarrow\N^2\\
n &\longmapsto 2\sigma(n) = (2i, 2j) \quad \text{where } \sigma(n) = (i,j).
\end{cases} \]
Note that for any injection $\sigma:\N\to\N^2$, we have $\pi_{\mathrm{R}}^{\sigma} = \pi_{\mathrm{R}}^{2\sigma}$.
Now let $\sigma_1, \sigma_2$ be any two injections such that $\pi_{\mathrm{R}}^{\sigma_1} = \pi_{\mathrm{R}}^{\sigma_2}$.
Then there exists a $\tau\in\monoline$ such that
\[ \tau\circ\sigma_1 = 2\sigma_2. \]
Indeed one such $\tau$ can be defined in the following way.
First let us define an injection $\phi:\N\to\N$, which will serve as a mapping for rows.
For any $i\in\N$, there are two possibilities:
\begin{itemize}
\item either there is a $j\in\N$ such that $(i,j)\in\im(\sigma_1)$, and then there is an even integer $i'\in\N$ such that $2\sigma_2 (\sigma_1^{-1}(i,j)) = (i',k)$ for some $k\in\N$.
This number $i'$ does not depend on $j$ because of the fact that $\pi_{\mathrm{R}}^{\sigma_1} = \pi_{\mathrm{R}}^{\sigma_2}$.
Indeed if $j_1,j_2\in\N$ are such that $(i,j_1),(i,j_2)\in \im(\sigma_1)$, then by definition $\sigma^{-1}(i,j_1)$ and $\sigma^{-1}(i,j_2)$ belong to the same block of $\pi_{\mathrm{R}}^{\sigma_1} = \pi_{\mathrm{R}}^{\sigma_2}$, and so $\sigma_2(\sigma^{-1}(i,j_1))$ and $\sigma_2(\sigma^{-1}(i,j_2))$ belong to the same block of $\pi_{\mathrm{R}}$.
So in that case we can define $\phi(i) := i'$.
\item either $\im(\sigma_1)\cap\{(i,j), \; j\geq 1\} = \emptyset$, and then we define $\phi(i) = 2i-1$.
\end{itemize}
The map $\phi$ is a well-defined injection, and we may now define
\[ \tau : \begin{cases}
(i,j) \in \im(\sigma_1) &\longmapsto 2 \sigma_2(\sigma_1^{-1}(i,j))\\
(i,j) \notin \im(\sigma_1) &\longmapsto (\phi(i), 2j-1)
\end{cases} \]
It is easy to check that $\tau \in\monoline$ and that $\tau\circ\sigma_1 = 2\sigma_2$.
We can now fix $\mu$ a $\monoline$-exchangeable measure on $\partitemb{\infty}$.
Consider a partition $\pi_0\in\partit{\infty}$ and an injection $\sigma_0 :\N\to\N^2$ such that $\pi_{\mathrm{R}}^{\sigma_0} = \pi_0$.
Now for any other $\sigma_1$ such that $\pi_{\mathrm{R}}^{\sigma_1} = \pi_0$, let $\tau\in\monoline$ be such that $\tau\circ\sigma_1 = 2\sigma_0$.
By $\monoline$-invariance of $\mu$, we have
\begin{align*}
\mu(\pi^{\sigma_1}\in\point) &= \mu(\pi^{\tau\circ\sigma_1}\in\point)\\
&= \mu(\pi^{2\sigma_0}\in\point).
\end{align*}
Therefore this measure does not depend on $\sigma_1$ but only on $\pi_0$, so we may define
\[ \kappa_{\pi_0} := \mu(\pi^{\sigma_0} \in\point), \]
which is a measure on $\partitemb{\infty}$, for all $\pi_0$.
Now it remains to check that for any injection $\sigma:\N\to\N$, we have $\kappa_{\pi_0}^{\sigma} = \kappa_{\pi_0^{\sigma}}$.
But if $\pi_{\mathrm{R}}^{\sigma_0} = \pi_0$, then $\pi_{\mathrm{R}}^{\sigma_0 \circ \sigma} = \pi_0^{\sigma}$, so
\begin{align*}
\kappa_{\pi_0}^{\sigma} &= \mu((\pi^{\sigma_0})^{\sigma} \in\point)\\
&= \mu(\pi^{\sigma_0\circ\sigma} \in\point)\\
&= \kappa_{\pi_0^{\sigma}},
\end{align*}
so $\kappa$ is a strongly exchangeable kernel from $\partit{\infty}$ to $\partitemb{\infty}$, and it is easy to check that the $\monoline$-invariant measure associated to $\kappa$ is $\mu$.
\end{proof}

Putting together Proposition~\ref{prop:outer_branching} and Proposition~\ref{prop:kappa_to_mu} gives us:

\begin{theorem} \label{thm:outer_charac}
Let $\Pi = (\Pi(t), t\geq 0) = ((\zeta(t), \xi(t)), t\geq 0)$ be a strongly exchangeable Markov process with values in $\partitemb{\infty}$ and decreasing càdlàg sample paths.
Suppose that $\Pi$ satisfies the \textbf{outer branching property}.
Then the distribution of $\Pi$ is characterized by an $\monoline$-invariant measure $\mu$ on $\partitemb{\N^2}$ satisfying
\begin{equation}\label{eq:mu_sig_finite}
\begin{gathered}
\mu\left (\pi\nprec(\pi_{\mathrm{R}}, \mathbf{1})\right ) = 0\\
\text{and }\forall n\in\N, \qquad \mu\left (\pi_{|[n]^2} \neq (\pi_{\mathrm{R}}, \mathbf{1})_{|[n]^2}\right ) < \infty.
\end{gathered}
\end{equation}
The characterization is in the sense that for any $\pi_0,\pi_1 \in \partit{\infty}$ with an infinity of infinite blocks,
\[ \mu = \kappa_{\pi_0}^{\sigma_0}\quad \text{ and }\quad \kappa_{\pi_1} = \mu^{\sigma_1}, \]
where $\kappa$ is the simplified characteristic kernel of $\Pi$, $\sigma_0:\N^2\to\N$ is any injection such that $\pi_{0}^{\sigma_0} = \pi_{\mathrm{R}}$ and $\sigma_1:\N\to\N^2$ is any injection such that $\pi_{\mathrm{R}}^{\sigma_1} = \pi_1$.

Conversely, for any such measure $\mu$, there is a strongly exchangeable Markov process with values in $\partitemb{\infty}$, decreasing càdlàg sample paths and the outer branching property with characteristic measure $\mu$.
\end{theorem}

\begin{remark}
An explicit construction for the converse part of the theorem is described in the next section 
(Lemma~\ref{lemma:poisson_construct}).
\end{remark}

\subsection{Poissonian construction} \label{sec:poisson_construct}

Consider $\mu$ an $\monoline$-invariant measure on $\partitemb{\N^{2}}$ satisfying~\eqref{eq:mu_sig_finite} and let $\Lambda$ be a Poisson point process on $\N\times\Rplus\times\partitemb{\N^{2}}$ with intensity $\#\otimes\dd t\otimes\mu$, where $\#$ denotes the counting measure and $\dd t$ the Lebesgue measure.

Fix $n\in\N$.
Because of~\eqref{eq:mu_sig_finite}, the points $(k,t,\pi)\in\Lambda$ such that $k\leq n$ and $\pi_{|[n]^2} \neq (\pi_{\mathrm{R}}, \mathbf{1})_{|[n]^2}$ can be numbered
\[ (k^n_i, t^n_i, \pi^n_i, i\geq 1) \quad \text{with } t^n_1 < t^n_2 < \ldots \quad \text{and } t^n_i \xrightarrow[i\to\infty]{} \infty. \]
Fix any initial value $\pi_0 \in\partitemb{\infty}$.
Let us define a process $(\Pi^n_i, i\geq 0)$ with values in $\partitemb{[n]}$, by $\Pi^n_0 = (\pi_0)\restr{n}$ and by induction, conditional on $\Pi^n_i = (\zeta, \xi)$:
\begin{itemize}
\item if $\xi$ has less than $k^n_{i+1}$ blocks, then set $\Pi^n_{i+1} := \Pi^n_i$
\item if $\xi$ has a $k^n_{i+1}$-th block, say $B$, then let $\tau:B\to [n]^2$ be the injection such that $\tau(k) = (i,j)$ iff $k\in B$ is the $j$-th element of the $i$-th block of $\zeta_{|B}$. 

Then define $\Pi^n_{i+1}$ as the only element $\pi\in\partitemb{n}$ such that $\pi\preceq \Pi^n_i$, $\pi_{|B} = (\pi^n_i)^{\tau}$ and $\pi_{|[n]\setminus B} = (\Pi^n_i)_{|[n]\setminus B}$.
\end{itemize}
Now we define the continuous-time processes $(\Pi^n(t), t\geq 0)$ by
\[ \Pi^n(t) := \Pi^n_i \quad\text{ iff } t\in [t^n_{i-1}, t^n_i). \]

\begin{lemma} \label{lemma:poisson_construct}
The processes $\Pi^n$ built from this Poissonian construction are consistent in the sense that we have for all $m\geq n\geq 1$ and $t\geq 0$,
\[ \Pi^m(t)\restr{n} = \Pi^n(t). \]
Therefore, for all $t\geq 0$, there is a unique random variable $\Pi(t)$ with values in $\partitemb{\infty}$ such that $\Pi(t)\restr{n} = \Pi^n(t)$ for all $n$, and the process $(\Pi(t), t\geq 0)$ is a strongly exchangeable decreasing Markov process with the outer branching property whose characteristic $\monoline$-invariant measure is $\mu$.
\end{lemma}

\begin{proof}
Choose a number $n\in\N$ and consider the variable $(k^{n+1}_1, t^{n+1}_1, \pi^{n+1}_1)$.
It is clear from the definition that $(\Pi^{n+1}_0)\restr{n} = \Pi^n_0$.
Now let us show that $(\Pi^{n+1}_1)\restr{n} = \Pi^n(t^{n+1}_1)$.


We distinguish two cases:\\
\textbf{1)} If $t^{n+1}_1 = t^n_1$, then we have necessarily $k^{n+1}_1 =k^{n}_1 \leq n$ and $(\pi^{n+1}_1)_{|[n]^2} =(\pi^{n}_1)_{|[n]^2} \neq (\pi_{\mathrm{R}}, \mathbf{1})_{|[n]^2}$.
Let us write $\Pi^{n+1}_0 = (\zeta^{n+1}, \xi^{n+1})$ and $\Pi^n_0 = (\zeta^n, \xi^n)$.
Since $(\Pi^{n+1}_0)\restr{n} = \Pi^n_0$, it is clear that the $k^{n}_1$-th block of $\xi^{n+1}$ includes the $k^{n}_1$-th block of $\xi^n$, and may at most contain one other element, the number $n+1$.
In other words we have
\[ B^{n+1}\cap [n] = B^n, \]
where $B^{n+1}$ and $B^n$ denote those two blocks.
Now let us write $\tau^{n+1}, \tau^n$ for the respective injections in $\N^2$ defined in the construction.
Because we defined the injections according to the ordering of the blocks of $\zeta$ and with the natural order on $\N$, it should be clear that we have
\[ \tau^{n+1}_{|B^n} = \tau^{n}. \]
Therefore we deduce $((\pi^{n}_1)^{\tau^{n+1}})_{|B^{n}} = (\pi^{n}_1)^{\tau^n}$, which allows us to conclude $(\Pi^{n+1}_1)\restr{n} = \Pi^n_1 = \Pi^n(t^{n+1}_1)$.

\hspace*{-1.5ex}\textbf{2)} If $t^{n+1}_1 < t^n_1$, then we have to further distinguish two possibilities:\vspace{0.5ex} \\
{
\leftskip=1.5ex
\textbf{a)} $k^{n+1}_1 = n+1$.
In that case the $n+1$-th block of $\xi^{n+1}$ can either be empty or the singleton $\{n+1\}$.
Then by definition, we necessarily have $\Pi^{n+1}_1 = \Pi^{n+1}_0$, so we can conclude $(\Pi^{n+1}_1)\restr{n} = \Pi^{n}_0 = \Pi^n(t^{n+1}_1)$.\vspace{0.5ex} \\
\textbf{b)} $(\pi^{n+1}_1)_{|[n]^2} = (\pi_{\mathrm{R}}, \mathbf{1})_{|[n]^2}$.
In that case, let $B$ be the $k^{n+1}_1$-th block of $\xi$ and $\tau:B\to[n+1]^2$ the injective map defined in the construction.
By definition, we have $(\pi_{\mathrm{R}}, \mathbf{1})^{\tau} = (\zeta,\xi)_{|B}$.
Also by definition of $\tau$, for any $k\leq n$, we have $\tau(k) \in [n]^2$.
Therefore, we can conclude that
\[ ((\pi^{n+1}_1)^{\tau})_{|B\cap[n]} = ((\pi^{n+1}_1)_{|[n]^2})^{\tau_{|B\cap[n]}} = (\pi_{\mathrm{R}}, \mathbf{1})^{\tau_{|B\cap[n]}} = (\zeta,\xi)_{|B\cap[n]}. \]
This shows that $(\Pi^{n+1}_1)\restr{n} = (\Pi^{n+1}_0)\restr{n}$, which allows us to conclude $(\Pi^{n+1}_1)\restr{n} = \Pi^n_0 = \Pi^n(t^{n+1}_1)$.
}

Note that by induction and the strong Markov property of the Poisson point process $\Lambda$, this proves that $(\Pi^{n+1}_i)\restr{n} = \Pi^n(t^{n+1}_i)$ for all $i\geq 1$, so $\Pi^{n+1}(t)\restr{n} = \Pi^{n}(t)$ for all $t\geq 0$, which concludes the first part of the proof.

It remains to show that the process $(\Pi(t), t\geq 0)$ is a strongly exchangeable Markov process with the outer branching property, and whose characteristic $\monoline$-invariant measure is $\mu$.

First, notice that from the construction, we deduce immediately that for any $n$, $\Pi^n$ is a Markov chain, and at any jump time $t^n_i$, the partitions $\Pi^n_{i-1}$ and $\Pi^n_{i}$ differ at most on one block of $\xi$, where $\Pi^n_{i-1} = (\zeta, \xi)$.
Therefore the distribution of the Markov chain $\Pi^n$ is given by the transition rates of the form
\[ q^n_{\pi_0, \pi_1}, \]
with $\pi_0=(\zeta, \xi)\in\partitemb{\infty}$, and with $\pi_1 \preceq (\pi_0)\restr{n}$ such that, for some $B\in\xi\restr{n}$, $(\pi_1)_{|[n]\setminus B} = (\pi_0)_{|[n]\setminus B}$ and $(\pi_1)_{|B} \prec (\pi_0)_{|B}$.
Now for such $\pi_0, \pi_1$, write $\tau:B\to\N^2$ for the injection such that $\tau(k) = (i,j)$ iff $k$ is the $j$-th element of the $i$-th block of $\zeta_{|B}$.
By elementary properties of Poisson point processes we have 
\begin{equation}\label{eq:qtomu}
 q^n_{\pi_0, \pi_1} = 
\mu\left ((\pi_{|\im\tau})^{\tau} = (\pi_1)_{|B}\right ).
\end{equation}
This implies that $\Pi$ is a strongly exchangeable Markov process whose characteristic $\monoline$-invariant measure is $\mu$.
Indeed, recall from Section~\ref{sec:proj_markov} that since $\Pi$ satisfies the projective Markov property and is exchangeable (this is immediate from the $\monoline$-invariance of $\mu$), $\Pi$ is strongly exchangeable, with a characteristic kernel $K$ such that with the same notation as in~\eqref{eq:qtomu},
\begin{equation}\label{eq:Ktoq}
 K_{\pi_0}(\pi\restr{n} = \pi_1) = q^{n}_{\pi_0,\pi_1}.
\end{equation}
Now the outer branching property is immediately deduced from the construction of the process, where it is clear that at any jump time, at most one block of the coarser partition is involved.
Therefore by Proposition~\ref{prop:outer_branching}, the law of $\Pi$ is characterized by the simpler kernel $\kappa$ defined by $\kappa_{\zeta} = K_{(\zeta,\mathbf{1})}$, for $\zeta\in\partit{\infty}$.
Now putting this together with~\eqref{eq:Ktoq} and~\eqref{eq:qtomu}, since the coarsest partition $\mathbf{1}$ only contains one block $B=\N$, we have simply
\[ \kappa_{\zeta}(\pi\restr{n} = \pi_1) = \mu\left ((\pi^{\tau})\restr{n} = \pi_1\right ), \]
where $\tau$ is an injection such that $\pi_{\mathrm{R}}^{\tau} = \zeta$.
In other words with these definitions we have $\kappa_{\zeta} = \mu^{\tau}$ which shows that $\mu$ is the characteristic $\monoline$-invariant measure of the process $\Pi$.
\end{proof}

\section{Inner branching property, simple fragmentations} \label{sec:inner_branch}

In this section we consider simple fragmentation processes, that is we will assume both branching properties.
This will allow us to further the analysis of the $\monoline$-invariant measure $\mu$ which appears in Theorem~\ref{thm:outer_charac}.
To introduce the next theorem and main result of this article, let us first give some examples of simple nested fragmentation processes.

\subsection{Some examples}\label{sec:examples}

\paragraph{Pure erosion}

For $i\geq 1$, let $\xi_{\mathrm{out}}^{(i)}$ be the partition of $\N^2$ with two blocks such that one of them is the $i$-th line $\{i\}\times\N$, i.e.\
\[ \xi_{\mathrm{out}}^{(i)}:=\big\{\{i\}\times\N, \;\N^{2}\setminus (\{i\}\times\N)\big\} \]
and define the outer erosion measure $\mathfrak{e}^{\mathrm{out}} := \sum_{i\geq 1}\delta(\pi_{\mathrm{R}},\xi_{\mathrm{out}}^{(i)})$, where for readability we denote without subscripts $\delta(\zeta,\xi)$ the Dirac measure on $(\zeta,\xi)$.

Similarly, for $i,j\geq 1$, we define
\begin{gather*}
\zeta_{\mathrm{in}}^{(i,j)} := \big\{\{(i,j)\}\big\} \cup \big\{(\{i\}\times\N)\setminus\{(i,j)\}\big\} \cup \big\{ \{k\}\times\N, \,k\geq 1, k\neq i\big\}, \\
\xi_{\mathrm{in}}^{(i,j)} := \big\{ \{(i,j)\}, \; \N^{2}\setminus\{(i,j)\} \big\},
\end{gather*}
and the inner erosion measures 
\[\mathfrak{e}^{\mathrm{in},1} := \sum_{i,j\geq 1}\delta(\zeta_{\mathrm{in}}^{(i,j)},\mathbf{1}) \quad \text{ and } \quad \mathfrak{e}^{\mathrm{in},2} := \sum_{i,j\geq 1}\delta(\zeta_{\mathrm{in}}^{(i,j)},\xi_{\mathrm{in}}^{(i,j)}).\]

Now, given three real numbers $c_{\mathrm{out}}, c_{\mathrm{in},1}, c_{\mathrm{in},2}\geq 0$, the $\monoline$-invariant measure $\mu=c_{\mathrm{out}}\mathfrak{e}^{\mathrm{out}}+c_{\mathrm{in},1}\mathfrak{e}^{\mathrm{in},1}+c_{\mathrm{in},2}\mathfrak{e}^{\mathrm{in},2}$ clearly satisfies~\eqref{eq:mu_sig_finite}, so by Theorem~\ref{thm:outer_charac} there exists a fragmentation process having $\mu$ as $\monoline$-invariant measure.

From the construction, we see that the rates of such a process can be described informally as follows:
\begin{itemize}
\item any inner block erodes out of its outer block at rate $c_{\mathrm{out}}$, i.e.\ it does not fragment but forms, on its own, a new outer block.
\item any integer erodes out of its inner block at rate $c_{\mathrm{in},1}$, forming a singleton inner block, within the same outer block as its parent.
\item any integer erodes out of its inner and outer block at rate $c_{\mathrm{in},2}$, forming singleton inner and outer blocks.
\end{itemize}
See Figure~\ref{fig:erosion} for a schematic representation of each erosion event.

\def\figheight{10em} 

\begin{figure}[ht]
\subcaptionbox{Outer erosion \label{fig:erosion_sub_out}}{\includegraphics[height=\figheight]{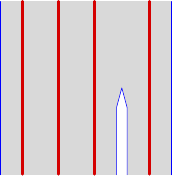}}
\hfill
\subcaptionbox{Inner erosion \label{fig:erosion_sub_in1}}{\includegraphics[height=\figheight]{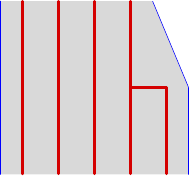}}
\hfill
\subcaptionbox{Inner erosion with creation of a new species \label{fig:erosion_sub_in2}} {\includegraphics[height=\figheight]{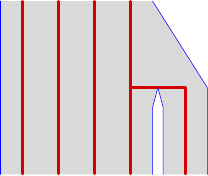}}
\captionsetup{justification=raggedright}
\caption{Pure erosion events.
  In each figure a unique branching event is shown, where the top part of figures represents an outer block -- depicted by a gray ``tube'' -- containing four inner blocks -- the four red lines within that tube -- and the bottom part represents the outcome of the event, to be understood as: \\
  \smallskip(\subref{fig:erosion_sub_out}) the fourth inner block -- call it $B$ -- erodes out from the outer block, creating a new outer block also equal to $B$; \\
  \smallskip(\subref{fig:erosion_sub_in1}) a singleton -- say $i\in B$ -- of the fourth inner block erodes out of $B$, creating a fifth inner block $\{i\}$, but the outer block remains unchanged; \\
  \smallskip(\subref{fig:erosion_sub_in2}) a singleton -- say $i\in B$ -- of the fourth inner block erodes out of both its inner and outer block thus forming two singleton inner and outer blocks equal to $\{i\}$.
} \label{fig:erosion}
\captionsetup{justification=centering}
\end{figure}

\paragraph{Outer dislocation}

Recall the definition of the space of mass partitions $\mathbf{s}=(s_1, s_2, \ldots)\in\masspart$ and of the measures $\rho_{\mathbf{s}}$ from Section~\ref{sec:univariate_results}.
We define in a similar way, a collection of probability measure $\widehat{\rho}_{\mathbf{s}}$ on $\partitemb{\infty}$, by constructing $\pi=(\zeta,\xi)\sim\widehat{\rho}_{\mathbf{s}}$ with the following so-called paintbox procedure:
\begin{itemize}
\item for $k\geq 0$, let $t_k := \sum_{k'=1}^{k}s_{k'}$, with $t_0=0$ by convention.
\item let $U_1,U_2,\ldots$ be a sequence of i.i.d.\ uniform r.v.\ on $[0,1]$ and define the random partition $\xi$ on $\N^{2}$ by
\[ (i,j)\sim^{\xi}(i',j') \iff i=i' \text{ or } U_i,U_{i'} \in [t_k, t_{k+1})\text{ for a unique }k\geq 0. \]
\item $\widehat{\rho}_{\mathbf{s}}$ is now defined to be the distribution of the random nested partition $\pi=(\pi_{\mathrm{R}},\xi)$.
\end{itemize}
Now for $\nu_{\mathrm{out}}$ a measure on $\masspart$ satisfying~\eqref{eq:nu_condition_simple}, we define
\[ \widehat{\rho}_{\nu_{\mathrm{out}}}(\point) := \int_{\masspart}\nu_{\mathrm{out}}(\dd \mathbf{s}) \,\widehat{\rho}_{\mathbf{s}}(\point). \]

It is straight-forward to check that $\widehat{\rho}_{\nu_{\mathrm{out}}}$ is an $\monoline$-invariant measure measure on $\partitemb{\N^{2}}$ satisfying~\eqref{eq:mu_sig_finite}, so there exists a fragmentation process having $\widehat{\rho}_{\nu_{\mathrm{out}}}$ as $\monoline$-invariant measure.

In intuitive terms, such a process can be described by saying that the outer blocks independently dislocate \emph{around their inner blocks} with \textbf{outer dislocation rate} $\nu_{\mathrm{out}}$.
In a dislocation event, inner blocks are unchanged, and they are indistinguishable.
By construction, each newly created outer block ``picks'' a given frequency of inner blocks among those forming the original outer block (see Figure~\ref{fig:dout}).

\begin{figure}[ht]
\centering
\includegraphics[height=\figheight]{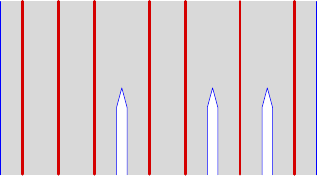}
\captionsetup{justification=raggedright}
\caption{Outer dislocation event.
  Here the initial outer block $B$ contains seven inner blocks, and at the branching event splits into new outer blocks $B_i, i \geq 1$, to which a mass $s_i\geq 1$ is assigned.
  The inner blocks are unchanged and each of them independently picks a newly created outer block with probability $s_i$, or create an outer block equal to itself with probability $1-\sum_i s_i$.
} \label{fig:dout}
\captionsetup{justification=centering}
\end{figure}

\paragraph{Inner dislocation}

The upcoming example is the more complex on our list, exhibiting simultaneous inner and outer fragmentations.
However, in construction it is very similar to the previous example, and it should pose no difficulties to get a good intuition of the dislocation mechanics.

Let us first formally define a space which will serve as an analog of the space of mass partitions $\masspart$.


\begin{definition}\label{def:masspartemb}
We define a particular space of ``bivariate mass partitions''
\[ \masspartemb \subset [0,1]^{\N}\times[0,1]^{\N^2}\times[0,1]\times[0,1]^{\N} \]
as the subset consisting of elements $\mathbf{p} = ((u_l)_{l\geq 1}, (s_{k,l})_{k,l\geq 1}, \bar{u}, (\bar{s}_k)_{k\geq 1})$ satisfying the following conditions.
\begin{equation}\label{eq:def_masspartemb}
\begin{gathered}
u_1\geq u_2\geq \ldots \text{ and }\textstyle\sum_{l}u_l \leq \bar{u},\\
\forall k\geq 1, \, s_{k,1}\geq s_{k,2}\geq \ldots \text{ and } \textstyle\sum_{l}s_{k,l} \leq \bar{s}_k,\\
\bar{s}_1 \geq \bar{s}_2 \geq \ldots,\\
\bar{u} + \textstyle\sum_{k}\bar{s}_k \leq 1,\\
\text{if }\bar{s}_k = \bar{s}_{k+1}, \text{ then } (l_0 = \inf\{l\geq 1, s_{k,l}\neq s_{k+1, l}\} <\infty) \Rightarrow (s_{k,l_0} > s_{k+1, l_0}).
\end{gathered}
\end{equation}
\end{definition}
Note that $\masspartemb$ is a compact space with respect to the product topology since it is a closed subset of the compact space $[0,1]^{\N}\times[0,1]^{\N^2}\times[0,1]\times[0,1]^{\N}$.
Therefore considering this topology, we will have no trouble considering measures on $\masspartemb$.

Now, given a fixed $i\geq 1$ and $\mathbf{p} = \left ((u_l)_{l\geq 1}, (s_{k,l})_{k,l\geq 1}, \bar{u}, (\bar{s}_k)_{k\geq 1}\right )\in\masspartemb$, one can define a random element $\pi^{(i)} = (\zeta^{(i)}, \xi^{(i)}) \in\partitemb{\N^{2}}$ with the following paintbox procedure: \label{page:paintbox_proc}
\begin{itemize}
\item for $k\geq 0$, define $\bar{t}_k = \bar{u} + \sum_{k'=1}^{k} \bar{s}_{k'}$.
\item for $l\geq 0$, define $t_{\star,l} = \sum_{l'=1}^{l} u_{l'}$.
\item for $k\geq 1$ and $l\geq 0$, define $t_{k,l} = \bar{t}_{k-1} + \sum_{l'=1}^{l} s_{k,l'}$.
\item write $\pi_0 = (\zeta_0, \xi_0)$ for the unique element of $\partitemb{[0,1]}$ such that the non-dust blocks of $\xi_0$ are
\[ [0,\bar{u}) \text{ and } [\bar{t}_{k-1}, \bar{t}_{k}), \; k\geq 1, \]
and such that the non-singleton blocks of $\zeta_0$ are
\[ [t_{\star, l-1},t_{\star, l}),\; l\geq 1 \text{ and } [t_{k,l-1}, t_{k,l}), \; k,l\geq 1. \]
\item let $(U_j, j\geq 1)$ be a i.i.d.\ sequence of uniform r.v.\ on $[0,1]$.
\item define the random element $\pi^{(i)}\in\partitemb{\N^{2}}$ as the unique element $\pi^{(i)} = (\zeta^{(i)}, \xi^{(i)}) \preceq (\pi_{\mathrm{R}}, \mathbf{1})$ such that
\begin{itemize}
\item $(\zeta^{(i)}, \xi^{(i)})_{|(\N\setminus\{i\})\times\N} = (\pi_{\mathrm{R}}, \mathbf{1})_{|(\N\setminus\{i\})\times\N}$, i.e.\ only the $i$-th row may dislocate.
\item On the $i$-th row, we have
\[ (i,j)\sim^{\zeta^{(i)}}(i,j') \iff U_j \sim^{\zeta_0} U_{j'}, \]
\[ (i,j)\sim^{\xi^{(i)}}(i,j') \iff U_j \sim^{\xi_0} U_{j'}, \]
and also
\[ (i,j)\sim^{\xi^{(i)}}(i+1,1) \iff U_j \in [0,\bar{u}),  \]
where it should be noted that $(i+1,1)$ is an element of the block of $\xi^{(i)}$ that contains $(\N\setminus\{i\})\times\N$.
\end{itemize}
See figure \ref{fig:paintbox} for a representation of the bivariate paintbox process.
In words, $\pi^{(i)}$ is a random nested partition such that the outer partition $\xi^{(i)}$ has a ``distinguished block'' containing $(\N\setminus\{i\})\times\N$, which also contains a proportion $\bar{u}$ of elements of the $i$-th row.
Other non-singleton blocks of $\xi^{(i)}$ can be indexed by $k\geq 1$, each containing a proportion $\bar{s}_k$ of elements of the $i$-th row.
The blocks of the inner partition $\zeta^{(i)}$ are the entire rows, except for the $i$-th row where non-singleton blocks can be indexed by $(\star, l)$ and $(k,l)$ for $k,l\geq 1$, each respectively containing a proportion $u_l$ or $s_{k,l}$ of elements of the $i$-th row.
As the notation suggests, inner blocks with frequency $s_{k,l}$ (resp.\ $u_l$) are included in the outer block with frequency $\bar{s}_k$ (resp.\ $\bar{u}$) on the $i$-th row.
%
\end{itemize}
\begin{figure}[ht]
  %
  %
  \centering
  \includegraphics[width=.9\linewidth]{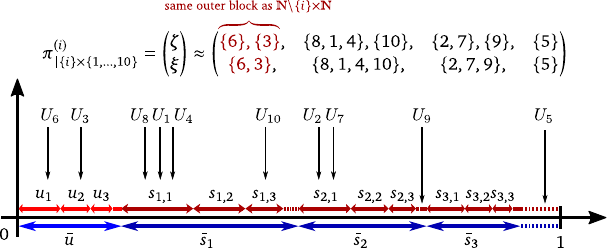}
  \caption{Paintbox construction of $\pi^{(i)}$}
  \label{fig:paintbox}
\end{figure}

The distribution of $\pi^{(i)}$ obtained with this construction is a probability on $\partitemb{\N^{2}}$ that we denote $\widetilde{\rho}^{(i)}_{\mathbf{p}}$.
We finally define
\[ \widetilde{\rho}_{\mathbf{p}} = \sum_{i\geq 1} \widetilde{\rho}^{(i)}_{\mathbf{p}}. \]
It should be clear from the exchangeability of the sequence $(U_j, j\geq 1)$ that $\widetilde{\rho}_{\mathbf{p}}$ is $\monoline$-invariant.

Now consider a measure $\nu_{\mathrm{in}}$ on $\masspartemb$ satisfying
\begin{equation}\label{eq:nu_condition}
\nu_{\mathrm{in}}(\{\mathbf{1}\}) = 0, \text{ and }\; \int_{\masspartemb} (1 - u_1) \, \nu_{\mathrm{in}}(\dd \mathbf{p}) < \infty,
\end{equation}
where $\mathbf{1}\in \masspartemb$ is defined as the unique element with $u_1=1$.
Similarly as in the previous example, we define
\[ \widetilde{\rho}_{\nu_{\mathrm{in}}}(\point) = \int_{\masspartemb}\widetilde{\rho}_{\mathbf{p}}(\point) \,\nu_{\mathrm{in}}(\dd \mathbf{p}). \]

It is again straight-forward to check that $\widehat{\rho}_{\nu_{\mathrm{in}}}$ is an $\monoline$-invariant measure measure on $\partitemb{\N^{2}}$ satisfying~\eqref{eq:mu_sig_finite}, so there exists a fragmentation process having $\widehat{\rho}_{\nu_{\mathrm{in}}}$ as $\monoline$-invariant measure.

In intuitive terms (see Figure~\ref{fig:din} for a picture), such a process can be described by saying that the inner blocks independently dislocate with \textbf{inner dislocation rate} $\nu_{\mathrm{in}}$.
In a dislocation event, new inner blocks are formed, each with a given proportion of the original block, and regroup, either in the original outer block (with a total proportion $\bar{u}$ with respect to the original inner block) or in newly created outer blocks.

\begin{figure}[ht]
\centering
\includegraphics[height=\figheight]{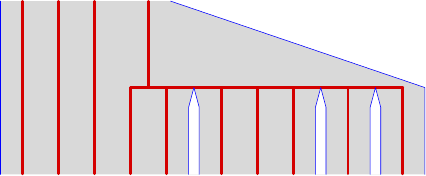}
\captionsetup{justification=raggedright}
\caption{Inner dislocation event.
  Here one of the initial inner blocks splits in blocks with frequencies given by $u_l, l\geq 1$ and $s_{k,l}, k,l \geq 1$.
  The blocks with frequencies $u_l$ remain in the original outer block, while for every $k$, the blocks with frequencies $s_{k,l}$ form a new outer block.
  There may be dust creation, in the original outer block with frequency $\bar u - \sum_lu_l$, in the newly created outer blocks with frequencies $\bar s_k - \sum_ls_{k,l}$ and dust outer blocks with frequency $1 - \bar u - \sum_k \bar s_k$.
} \label{fig:din}
\captionsetup{justification=centering}
\end{figure}


\paragraph{A combination of the above}

The mechanisms we discussed in the three proposed examples can be added in a parallel way, each event arising at its own independent rate and events from distinct mechanisms arising at distinct times.
More precisely, for a set of erosion coefficients $c_{\mathrm{out}},c_{\mathrm{in},1},c_{\mathrm{in},2}\geq 0$, an outer dislocation measure $\nu_{\mathrm{out}}$ on $\masspart$ satisfying~\eqref{eq:nu_condition_simple} and an inner dislocation measure $\nu_{\mathrm{in}}$ on $\masspartemb$ satisfying~\eqref{eq:nu_condition}, the measure
\[ \mu :=c_{\mathrm{out}}\mathfrak{e}^{\mathrm{out}}+c_{\mathrm{in},1}\mathfrak{e}^{\mathrm{in},1}+c_{\mathrm{in},2}\mathfrak{e}^{\mathrm{in},2} +\widehat{\rho}_{\nu_{\mathrm{out}}}+ \widetilde{\rho}_{\nu_{\mathrm{in}}}  \]
is a valid $\monoline$-invariant measure on $\partitemb{\N^{2}}$ satisfying~\eqref{eq:mu_sig_finite}, and thus corresponds to a fragmentation process exhibiting simultaneously all the discussed mechanisms at the rates described above.

The main result of this article is to prove that any nested simple fragmentation process admits such a representation.

\subsection{Characterization of simple nested fragmentations}

\begin{theorem}\label{thm:inner_charac} 
Let $\Pi = (\Pi(t), t\geq 0) = ((\zeta(t), \xi(t)), t\geq 0)$ be a strongly exchangeable Markov process with values in $\partitemb{\infty}$ and decreasing càdlàg sample paths.
Suppose that $\Pi$ is \textbf{simple}, that is it satisfies the outer and inner branching properties.
Then there are
\begin{itemize}
\item an outer erosion coefficient $c_{\mathrm{out}}\geq 0$ and two inner erosion coefficient $c_{\mathrm{in},1},c_{\mathrm{in},2}\geq 0$;
\item an outer dislocation measure $\nu_{\mathrm{out}}$ on $\masspart$ satisfying~\eqref{eq:nu_condition_simple};
\item an inner dislocation measure  $\nu_{\mathrm{in}}$ on $\masspartemb$ satisfying~\eqref{eq:nu_condition};
\end{itemize}
such that the $\monoline$-invariant measure $\mu$ of the process can be written
\[ \mu =c_{\mathrm{out}}\mathfrak{e}^{\mathrm{out}}+c_{\mathrm{in},1}\mathfrak{e}^{\mathrm{in},1}+c_{\mathrm{in},2}\mathfrak{e}^{\mathrm{in},2} +\widehat{\rho}_{\nu_{\mathrm{out}}}+ \widetilde{\rho}_{\nu_{\mathrm{in}}}  \]
\end{theorem}

The rest of Section~\ref{sec:inner_branch} consists in proving this result.

Let $ \mu $ be the $ \monoline $-invariant characteristic measure on $ \partitemb{\N^{2}} $ associated with $ \Pi $.
Recall that $ \pi_{\mathrm{R}} $ denotes the ``rows partition'', defined by
\[ \pi_{\mathrm{R}} = \big\{\{(i,j), \, j\geq 1\}, \; i\geq 1\big\}. \]
First, notice that the inner branching property implies that $ \mu $-a.e.\ we have
\[ \exists i \in \N, \quad \zeta_{|(\N\setminus\{i\}) \times \N} = (\pi_{\mathrm{R}})_{|(\N\setminus\{i\}) \times \N}, \]
where $\zeta$ is the first coordinate in the standard variable $\pi = (\zeta, \xi) \in\partitemb{\N^2}$.
This will enable us to decompose $ \mu $ further. 
Let us write
\begin{equation}\label{eq:def_mu_in_out}
\begin{gathered}
\mu_{\mathrm{out}} := \mu(\point\cap \{\zeta = \pi_{\mathrm{R}}\}),\\
\text{for }i\in\N,\quad \mu_{\mathrm{in},i} := \mu\left (\{\zeta_{|\{i\}\times \N} \neq \mathbf{1}_{\{i\} \times \N}\}\cap \point \right ), \\
\text{such that } \mu_{\mathrm{in}} := \mu(\point\cap \{\zeta \neq \pi_{\mathrm{R}}\}) = \textstyle\sum_{i\geq1} \mu_{\mathrm{in},i} \\
\text{and } \mu = \mu_{\mathrm{out}} + \mu_{\mathrm{in}}.
\end{gathered}
\end{equation}

On the event $ \{\zeta = \pi_{\mathrm{R}}\} $, we have 
\[ \xi = f(\xi^{\sigma}), \]
where $\sigma:\N\to\N^2$ is the injection $ i \mapsto (i,1) $, and $f:\partit{\infty}\to\partit{\N^2}$ is the map such that $(i,j) \sim^{f(\pi_0)} (i',j') \iff i\sim^{\pi_0}i'$.
By $\monoline$-invariance of $\mu$, the measure
\[ \widetilde{\mu}_{\mathrm{out}} := \mu\left (\{\zeta = \pi_{\mathrm{R}}\}\cap\{\xi^\sigma \in \point\}\right ) \]
is an exchangeable measure on $\partit{\infty}$, of which $\mu_{\mathrm{out}}$ is the push-forward by the application $(\pi_{\mathrm{R}}, f(\point))$.

Also, note that $\mu$ satisfies the $\sigma$-finiteness assumption~\eqref{eq:mu_sig_finite}, which implies that $\widetilde{\mu}_{\mathrm{out}}$ satisfies~\eqref{eq:mu_sig_finite_simple}, showing (see Section~\ref{sec:univariate_results}) that it can be decomposed
\[ \widetilde{\mu}_{\mathrm{out}} = c_{\mathrm{out}}\mathfrak{e} + \rho_{\nu_{\mathrm{out}}}, \]
where $c_{\mathrm{out}}\geq 0$ and $\nu_{\mathrm{out}}$ is a measure on $\masspart$ satisfying~\eqref{eq:nu_condition_simple}.
Thanks to our definitions, this immediately translates into
\[ \mu_{\mathrm{out}} = c_{\mathrm{out}}\mathfrak{e}^{\mathrm{out}}+\widehat{\rho}_{\nu_{\mathrm{out}}}, \]
and to prove Theorem~\ref{thm:inner_charac}, it only remains to show that we can write
\[ \mu_{\mathrm{in}} =\textstyle\sum_{i\geq1} \mu_{\mathrm{in},i} =c_{\mathrm{in},1}\mathfrak{e}^{\mathrm{in},1}+c_{\mathrm{in},2}\mathfrak{e}^{\mathrm{in},2} + \widetilde{\rho}_{\nu_{\mathrm{in}}}.  \]
To that aim, note that by exchangeability we have $\mu_{\mathrm{in},i} = \mu_{\mathrm{in},1}^{\tau_{1,i}}$ where $\tau_{1,i}:\N^2\to\N^2$ denotes the application swapping the first and $i$-th rows, so the application $\mu_{\mathrm{in},1}$ is sufficient to recover $\mu_{in}$ entirely.
Let us examine the distribution of $\xi$ under $\mu_{\mathrm{in},1}$.
We claim that $\mu$-a.e.\ on the event $\{\zeta_{|\{1\}\times \N} \neq \mathbf{1}_{\{1\} \times \N}\}$, that $\xi_{|(\N\setminus\{1\}) \times \N} = \mathbf{1}_{(\N\setminus\{1\}) \times \N}$.
Indeed, if this was not the case, we would have 
\[ a := \mu(\zeta_{|\{1\}\times \N} \neq \mathbf{1}_{\{1\} \times \N},\text{ and } (2,1) \nsim^{\xi} (3,1)) > 0. \]
Let us then show that in fact $a = 0$.
By $\monoline$-invariance of $\mu$, we have for any $i\geq 4$,
\[ a = \mu(\zeta_{|\{i\}\times \N} \neq \mathbf{1}_{\{i\} \times \N},\text{ and } (2,1) \nsim^{\xi} (3,1)), \]
but because of the inner branching property, we have seen that the events $\{\zeta_{|\{i\}\times \N} \neq \mathbf{1}_{\{i\} \times \N}\}$ have $\mu$-negligible intersections.
Now we have
\begin{align*}
\infty > \mu(\pi_{|[3]^2} \neq (\pi_{\mathrm{R}}, \mathbf{1})_{|[3]^2}) 
&\geq \mu\left ( (2,1) \nsim^{\xi} (3,1)\right )\\
&\geq \mu\left (\textstyle\cup_{i\geq 4}\{\zeta_{|\{i\}\times \N} \neq \mathbf{1}_{\{i\} \times \N},\text{ and } (2,1) \nsim^{\xi} (3,1)\}\right )\\
&= \textstyle \sum_{i\geq 4} a.
\end{align*}
This shows that necessarily $a = 0$.

Now in order to further study $\mu_{\mathrm{in},1}$ we need to introduce exchangeable partitions on a space with a distinguished element.
Results in that direction have been established by Foucart \cite{Fou11}, where distinguished exchangeable partitions are introduced and used to construct a generalization of $\Lambda$-coalescents modeling the genealogy of a population with immigration.
Here we need to define in a similar way distinguished partitions in our bivariate setting.
Informally, we will see that in a gene fragmentation, certain resulting gene blocks remain in a distinguished species block, that one can interpret as the mother species.
\begin{definition}
For $n\in\N\cup\{\infty\}$, we define $[n]_{\star} := [n]\cup\{\star\}$, where $\star$ is not an element of $\N$.
We define $\partitemb{n,\star}$ as the set of nested partitions $\pi=(\zeta, \xi)\in\partitemb{[n]_{\star}}$ such that $\star$ is isolated in the finer partition $\zeta$:
\[ \partitemb{n,\star} := \left \{ \pi=(\zeta, \xi)\in\partitemb{[n]_{\star}},\; \{\star\}\in\zeta\right  \}. \]
We define the action of an injection $\sigma : [n]\to[n]$ on an element $\pi\in\partitemb{n,\star}$ as the action of the unique extension $\widetilde \sigma : [n]_{\star}\to[n]_{\star}$ such that $\widetilde \sigma(\star) = \star$, and define \textbf{exchangeability} for measures on $\partitemb{n,\star}$ as invariance under the actions of such injections $\sigma:[n]\to[n]$.
\end{definition}

Let us come back to the decomposition of $\mu_{\mathrm{in},1}$.
We define an injection
\[ \tau : \begin{cases}
[\infty]_\star &\longrightarrow \N^2\\
j \in \N &\longmapsto (1,j)\\
\star &\longmapsto (2,1).
\end{cases} \]
Note that here we could have chosen any value $\tau(\star)=(i,j)$ with $i\geq 2$, since $\mu$-a.e.\ on the event $\{\zeta_{|\{1\}\times \N} \neq \mathbf{1}_{\{1\} \times \N}\}$ those elements are contained in the blocks of $\zeta$ which do not fragment.
In intuitive terms, the argument above shows that when the first gene block undergoes fragmentation, it may create new gene blocks that will be distributed in (possibly) new species blocks, and (possibly) the distinguished ``mother species'' block, which is the unique species block that will contain all of the original gene blocks which do not fragment.
For this reason, on the event $\{\zeta_{|\{1\}\times \N} \neq \mathbf{1}_{\{1\} \times \N}\}$, we have $\mu$-a.e.\ the equality
\[ \pi = (\zeta, \xi) = g(\pi^\tau), \]
where $g:\partitemb{\infty,\star} \to \partitemb{\N^2}$ is a deterministic function which we can define by: $g(\pi_0)$ is the only $\pi\in\partitemb{\N^2}$ such that
\begin{align*}
&\pi^{\tau} = \pi_0,\quad \pi \preceq (\pi_{\mathrm{R}}, \mathbf{1}_{\N^2}) \\
\text{and }&\pi_{|(\N\setminus\{1\})\times\N} = (\pi_{\mathrm{R}}, \mathbf{1}_{\N^2})_{|(\N\setminus\{1\})\times\N}.
\end{align*}
Let us now write 
\begin{equation}\label{eq:def_mu_in_tilde}
 \widetilde{\mu}_{\mathrm{in}} := \mu_{\mathrm{in},1}(\pi^{\tau}\in\point). 
\end{equation}
Note that the push-forward of this exchangeable measure on $\partitemb{\infty,\star}$ by the application $g$ is $\mu_{\mathrm{in},1}$.

Also, note that the $\sigma$-finiteness assumption~\eqref{eq:mu_sig_finite} implies that $\widetilde{\mu}_{\mathrm{in}}$ satisfies
\begin{equation}\label{eq:mu_sig_finite_star}
\forall n\geq 1,\quad \widetilde{\mu}_{\mathrm{in}}(\{\pi_{|[n]_{\star}} \neq \pi_n\}) < \infty,
\end{equation}
where $\pi_n:=(\{\{\star\},[n]\},\, \mathbf{1}_{[n]_{\star}})$ denotes the coarsest partition on $\partitemb{n,\star}$.

We can summarize the previous discussion in the following lemma.

\begin{lemma}\label{lem:intermediate}
The characteristic $\monoline$-invariant measure $\mu$ of a simple nested fragmentation process in $\partitemb{\infty}$ 
can be decomposed
\[ \mu = c_{\mathrm{out}}\mathfrak{e}^{\mathrm{out}}+\widehat{\rho}_{\nu_{\mathrm{out}}}+ \mu_{\mathrm{in}}, \]
where $c_{\mathrm{out}}\geq 0$, $\nu_{\mathrm{out}}$ is a measure on $\masspart$ satisfying~\eqref{eq:nu_condition_simple}, and $\mu_{\mathrm{in}} := \mu(\point\cap \{\zeta \neq \pi_{\mathrm{R}}\})$.
Also, there exists an exchangeable measure $\widetilde{\mu}_{\mathrm{in}}$ on $\partitemb{\infty, \star}$ such that $\mu_{\mathrm{in}} = \sum_{i}\mu_{\mathrm{in},1}^{\tau_{1,i}}$, where
\begin{itemize}
\item $\mu_{\mathrm{in},1}$ is a measure on $\partitemb{\N^2}$, satisfying~\eqref{eq:mu_sig_finite_star}, which is the push-forward of $\widetilde{\mu}_{\mathrm{in}}$ by the map $g$ defined in the previous paragraph.
\item $\tau_{1,i}:\N^2\to\N^2$ is the bijection swapping the first row with the $i$-th row.
\end{itemize}
\end{lemma}

In the next section, we will develop tools to analyze and further decompose the measure $\widetilde{\mu}_{\mathrm{in}}$ into terms of erosion and dislocation.

\subsection{Bivariate mass partitions}

Recall our compact space of bivariate mass partitions defined in Definition~\ref{def:masspartemb},
\[ \masspartemb \subset [0,1]^{\N}\times[0,1]^{\N^2}\times[0,1]\times[0,1]^{\N}, \]
as the subset consisting of elements $\mathbf{p} = ((u_l)_{l\geq 1}, (s_{k,l})_{k,l\geq 1}, \bar{u}, (\bar{s}_k)_{k\geq 1})$ satisfying conditions~\eqref{eq:def_masspartemb}.

We wish to match exchangeable measures on $\partitemb{\infty,\star}$ and measures on $\masspartemb$, and to that aim we need some further definition.
We say that an element $\pi=(\zeta, \xi)\in \partitemb{\infty,\star}$ has \textbf{asymptotic frequencies} if $\zeta$ and $\xi$ have asymptotic frequencies, and we write
\[ \abs{\pi}^{\downarrow} = ((u_l)_{l\geq 1}, (s_{k,l})_{k,l\geq 1}, \bar{u}, (\bar{s}_k)_{k\geq 1})\in\masspartemb \]
for the unique (because of the ordering conditions in~\eqref{eq:def_masspartemb}) element satisfying:
\begin{itemize}
\item the block $B\in\xi$ containing $\star$ has asymptotic frequency $\abs{B} = \bar{u}$ and the decreasing reordering of the asymptotic frequencies of the blocks of $\zeta\cap B$ is the sequence $({u_l, l\geq 1})$.
\item for any other block $B\in\xi$ with a positive asymptotic frequency, there is a $k\in\N$ such that $\abs{B} = \bar{s}_k$ and the decreasing reordering of the asymptotic frequencies of the blocks of $\zeta\cap B$ is the sequence $(s_{k,l}, l\geq 1)$.
\item the mapping $B\mapsto k$ is injective, and for any $k$ such that $\bar{s}_k>0$, there is a block $B\in\xi$ such that $\abs{B}=\bar{s}_k$.
\end{itemize}

\subsection{A paintbox construction for nested partitions}

We first adapt the construction used in our third example of Section~\ref{sec:examples} to our new partition space $\partitemb{\infty,\star}$.
Note that if $\mathbf{p} = \left ((u_l)_{l\geq 1}, (s_{k,l})_{k,l\geq 1}, \bar{u}, (\bar{s}_k)_{k\geq 1}\right )\in\masspartemb$, then one can define a random element $\pi = (\zeta, \xi) \in\partitemb{\infty, \star}$ with a paintbox procedure very similar as the one described as an example on p.~\pageref{page:paintbox_proc}.
For the sake of readability, let us recall the notation and construction.
\begin{itemize}
\item for $k\geq 0$, define $\bar{t}_k = \bar{u} + \sum_{k'=1}^{k} \bar{s}_{k'}$.
\item for $l\geq 0$, define $t_{\star,l} = \sum_{l'=1}^{l} u_{l'}$.
\item for $k\geq 1$ and $l\geq 0$, define $t_{k,l} = \bar{t}_{k-1} + \sum_{l'=1}^{l} s_{k,l'}$.
\item write $\pi_0 = (\zeta_0, \xi_0)$ for the unique element of $\partitemb{[0,1]}$ such that the non-dust blocks of $\xi_0$ are
\[ [0,\bar{u}) \text{ and } [\bar{t}_k, \bar{t}_{k+1}), \; k\geq 1, \]
and such that the non-singleton blocks of $\zeta_0$ are
\[ [t_{\star, l-1},t_{\star, l}),\; l\geq 1 \text{ and } [t_{k,l-1}, t_{k,l}), \; k,l\geq 1. \]
\item let $(U_i, i\geq 1)$ be a i.i.d.\ sequence of uniform random variables on $[0,1]$ and define the random injection $\sigma:i\in\N\mapsto U_i \in [0,1]$.
\item finally define the random element $\pi\in\partitemb{\infty,\star}$ as the unique $\pi = (\zeta, \xi)$ such that $\pi_{|\N} = \pi_0^{\sigma}$, and the block of $\xi$ containing $\star$ is equal to:
\[ \{\star\}\cup\{i\geq 1, \; U_i<\bar{u}\}. \]
\end{itemize}
The distribution of $\pi$ obtained with this construction is a probability on $\partitemb{\infty,\star}$ that we denote $\bar{\rho}_{\mathbf{p}}$.
It is clear from the exchangeability of the sequence $(U_i, i\geq 1)$ that $\bar{\rho}_{\mathbf{p}}$ is exchangeable, and from the strong law of large numbers, that $\bar{\rho}_{\mathbf{p}}$-a.s., $\pi$ possesses asymptotic frequencies equal to $\abs{\pi}^{\downarrow} = \mathbf{p}$.
For a measure $\nu$ on $\masspartemb$, we will define a corresponding exchangeable measure $\bar{\rho}_{\nu}$ on $\partitemb{\infty,\star}$ by
\[ \bar{\rho}_{\nu}(\point) = \int_{\masspartemb}\bar{\rho}_{\mathbf{p}}(\point) \,\nu(\dd \mathbf{p}). \]
The following lemma shows that every probability measure on $\partitemb{\infty,\star}$ is of this form.

\begin{lemma} \label{lem:partitemb_asymptotic_frequencies}
Let $\pi=(\zeta, \xi)$ be a random exchangeable element of $\partitemb{\infty,\star}$.
Then $\pi$ has asymptotic frequencies $\abs{\pi}^{\downarrow}\in\masspartemb$ a.s.\ and its distribution conditional on $\abs{\pi}^{\downarrow} = \mathbf{p}$ is $\bar{\rho}_{\mathbf{p}}$.
In other words, we have
\[ \P(\pi\in \point) = \int_{\masspartemb}\P(\abs{\pi}^{\downarrow}\in\dd \mathbf{p})\, \bar{\rho}_{\mathbf{p}}(\point). \]
\end{lemma}

\begin{proof}
Independently from $\pi$, let $(X_i, i\geq 1)$ and $(Y_i, i\geq 1)$ be i.i.d.\ uniform random variables on $[0,1]$.
Conditional on $\pi$, we define a random variable $Z_n \in[0,1]\times([0,1]\cup\{\star\}) $ for each $n\in\N$ by
\[ Z_n := \begin{cases}
 (X_{A_n}, Y_{B_n})&\text{if } \star\nsim^{\xi}n,\\
 (X_{A_n}, \star) &\text{if } \star\sim^{\xi}n,
\end{cases}\quad \text{ where } 
\begin{cases}
A_n := \min\{m \in \N, \;m\sim^{\zeta} n\} \\
B_n := \min\{m \in \N, \;m\sim^{\xi} n\}.
\end{cases} \]
It is straight-forward that we recover entirely $\pi$ from the sequence $(Z_n, n\geq 1)$ because we have
\begin{equation} \label{eq:pi_from_zn}
\begin{aligned} 
n\sim^{\zeta}m &\iff x(Z_n) = x(Z_m),\\
n\sim^{\xi}m &\iff y(Z_n) = y(Z_m),\\
n\sim^{\xi}\star &\iff y(Z_n) = \star,
\end{aligned}
\end{equation}
where $x$ and $y$ denote respectively the projection maps from $[0,1]\times([0,1]\cup\{\star\})$ to the first and second coordinates.
Now, notice that the exchangeability of $\pi$ implies that the sequence $(Z_n,n\geq 1)$ is an exchangeable sequence of random variables.
Then, by an application of De Finetti's theorem, we see that there is a random probability measure $P$ on $[0,1]\times\nobreak([0,1]\cup\{\star\})$ such that conditional on $P$, the sequence $(Z_n, n\geq 1)$ is i.i.d.\ distributed with distribution $P$.

Now notice that if $P$ is a probability measure on $[0,1]\times([0,1]\cup\{\star\})$, we can define
\[ \abs{P}^{\downarrow} = \left ((u_l)_{l\geq 1}, (s_{k,l})_{k,l\geq 1}, \bar{u}, (\bar{s}_k)_{k\geq 1}\right ) \in\masspartemb \]
by setting the following, where everything is numbered in an order compatible with our conditions~\eqref{eq:def_masspartemb}.
\begin{itemize}
\item $\bar{u} := P(y=\star)$.
\item $\bar{s}_k := P(y=y_k)$, where $(y_k, k\geq 1)$ is the injective sequence of points of $[0,1]$ such that $P(y=y_k)>0$.
\item $u_l := P(x=x_{\star,l}, y=\star)$ where $(x_{\star,l}, l\geq 1)$ is the injective sequence of points of $[0,1]$ such that $P(x=x_{\star,l}, y=\star)>0$.
\item $s_{k,l} := P(x=x_{k,l}, y=y_k)$ where $(x_{k,l}, l\geq 1)$ is the injective sequence of points of $[0,1]$ such that $P(x=x_{k,l}, y=y_k)>0$.
\end{itemize}
It should now be clear that defining with~\eqref{eq:pi_from_zn} a random $\pi\in\partitemb{\infty,\star}$ from a sequence $(Z_n,n\geq\nobreak 1)$ of $P$-i.i.d.\ random variables is in fact the same as defining $\pi$ from a paintbox construction $\bar{\rho}_{\mathbf{p}}$ with $\mathbf{p} = \abs{P}^{\downarrow}$.
Therefore, the distribution of $\pi$ is given by
\[ \P(\pi\in\point) = \int_{\masspartemb}\P(\abs{P}^{\downarrow}\in\dd\mathbf{p})\, \bar{\rho}_{\mathbf{p}}(\point), \]
which concludes the proof since for any $\mathbf{p}$ we have $\bar{\rho}_{\mathbf{p}}$-a.s. that $\abs{\pi}^{\downarrow}$ exists and is equal to~$\mathbf{p}$.
\end{proof}

\subsection{Erosion and dislocation for nested partitions}

As in the standard $\partit{\infty}$ case, we can decompose any exchangeable measure $\mu$ on $\partitemb{\infty,\star}$ satisfying some finiteness condition similar to \eqref{eq:mu_sig_finite_simple} in a canonical way.
To ease the notation, recall that we define for $n\in\N\cup\{\infty\}$, $\pi_n$ the maximal element in $\partitemb{n, \star}$
\[ \pi_n := (\{\{\star\},[n]\},\, \mathbf{1}_{[n]_{\star}}). \]
We also define two erosion measures $\mathfrak{e}^1$ and $\mathfrak{e}^2$ by
\begin{align*}
&\mathfrak{e}^1 = \sum_{i\geq 1} \delta_{(\{\{\star\},\{i\},[\infty]\setminus\{i\}\},\mathbf{1}_{[\infty]_{\star}})},\\
&\mathfrak{e}^2 = \sum_{i\geq 1} \delta_{(\{\{\star\},\{i\},[\infty]\setminus\{i\}\},\;\{\{i\},[\infty]_{\star}\setminus\{i\}\})}.
\end{align*}
Finally, we define $\mathbf{1} \in\masspartemb$ as the element $ \left ((u_l)_{l\geq 1}, (s_{k,l})_{k,l\geq 1}, \bar{u}, (\bar{s}_k)_{k\geq 1}\right )\in\masspartemb$ with $\bar{u} = u_1 = 1$ (note that $\bar{\rho}_{\mathbf{1}} = 
\delta_{\pi_\infty}$).

\begin{prop} \label{prop:mu_canonical_decomposition}
Let $\mu$ be an exchangeable measure on $\partitemb{\infty,\star}$ such that
\begin{equation}\label{eq:mu_sig_finite_inner}
\mu(\{\pi_\infty\}) = 0, \text{ and }\; \forall n\geq 1, \quad \mu(\pi_{|[n]_{\star}} \neq \pi_n) < \infty.
\end{equation}
Then there are two real numbers $c_1,c_2 \geq 0$ and a measure $\nu$ on $\masspartemb$ satisfying
\eqref{eq:nu_condition}, namely
\[
\nu(\{\mathbf{1}\}) = 0, \text{ and }\; \int_{\masspartemb} (1 - u_1) \, \nu(\dd \mathbf{p}) < \infty
\]
such that $\mu = c_1 \mathfrak{e}^1 + c_2 \mathfrak{e}^2 + \bar{\rho}_{\nu}$.
Conversely, any $\mu$ of this form is exchangeable and satisfies~\eqref{eq:mu_sig_finite_inner}.
\end{prop}

\begin{proof}
The proof follows closely that of Theorem 3.1 in \cite{bertoinRandom2006}, as our result is a straight-forward extension of it.
We first define $\mu_n := \mu(\point\cap\{\pi\restr{n}\neq\pi_n\})$ which is a finite measure, and
\[ \overleftarrow{\mu}_n := \mu_n^{\theta_n}, \]
where $\theta_n:\N\to\N$ is the $n$-shift defined by $\theta_n(i) = i+n$.
We can check that $\overleftarrow{\mu}_n$ is an exchangeable measure on $\partitemb{\infty, \star}$.
Indeed let us take $\sigma:\N\to\N$ a permutation, and consider $\tau:\N\to\N$ the permutation defined by
\[ \tau: \begin{cases}
i \leq n &\longmapsto i\\
i>n &\longmapsto n + \sigma^{-1}(i-n).
\end{cases} \]
We have clearly $\tau\circ\theta_n\circ\sigma = \theta_n $ and $\tau\restr{n} = \id_{[n]}$, so we can use the $\tau$-invariance of $\mu$ to conclude
\begin{align*}
\overleftarrow{\mu}_n(\pi^{\sigma} \in\point) &= \mu_n(\pi^{\theta_n\circ\sigma} \in\point)\\
&= \mu(\{\pi^{\theta_n\circ\sigma} \in\point\} \cap \{\pi_{|[n]_{\star}} \neq \pi_n\})\\
&= \mu(\{\pi^{\tau\circ\theta_n\circ\sigma} \in\point\} \cap \{(\pi^{\tau})_{|[n]_{\star}} \neq \pi_n\})\\
&= \mu(\{\pi^{\theta_n} \in\point\} \cap \{\pi_{|[n]_{\star}} \neq \pi_n\})\\
&= \overleftarrow{\mu}_n(\point),
\end{align*}
which proves that $\overleftarrow{\mu}_n$ is exchangeable.
Since it is also finite, Lemma~\ref{lem:partitemb_asymptotic_frequencies} implies that $\abs{(\pi^{\theta_n})}^{\downarrow} = \abs{\pi}^{\downarrow}$ exists $\mu$-a.e.\ on the event $\{\mu_{|[n]_\star} \neq \pi_n\}$, and that we have
\[ \overleftarrow{\mu}_n(\point) = \int_{\masspartemb}\mu_n(\abs{\pi}^{\downarrow}\in\dd \mathbf{p}) \; \bar{\rho}_{\mathbf{p}}(\point). \]
Now since $\cup_n \{\pi_{|[n]_\star} \neq \pi_n\} = \{\pi \neq \pi_\infty\}$ and $\mu(\{\pi = \pi_\infty\}) = 0$, we have necessarily the existence of $\abs{\pi}^{\downarrow}\in\masspartemb$ $\mu$-a.e.

Let us define $\phi(\point) := \mu(\point\cap\{\abs{\pi}^{\downarrow} \neq \mathbf{1}\})$.
Fix $k\in \N$, and consider the measure $\phi(\pi_{|[k]_{\star}} \in\point)$ on $\partitemb{k,\star}$.
We use the fact that $\mu(B) = \lim_{n\to\infty}\mu_n(B)$ for any Borel set $B\subset \partitemb{\infty,\star}$ and that $\mu_n(\abs{\pi}^{\downarrow} \in \point) = \overleftarrow{\mu}_n(\abs{\pi}^{\downarrow} \in \point)$ to write
\begin{align*}
\phi(\pi_{|[k]_{\star}} \in\point) &= \mu(\{\pi_{|[k]_{\star}} \in\point\}\cap\{\abs{\pi}^{\downarrow} \neq \mathbf{1}\})\\
&= \lim_{n\to\infty} \mu\left (\{\pi_{|[k]_{\star}} \in\point\}\cap\{\abs{\pi}^{\downarrow} \neq \mathbf{1},\; (\pi^{\theta_{k}})_{|[n]_\star} \neq \pi_n\}\right )\\
&= \lim_{n\to\infty} \mu\left (\{(\pi^{\theta_n})_{|[k]_{\star}} \in\point\}\cap\{\abs{\pi}^{\downarrow} \neq \mathbf{1},\; \pi_{|[n]_\star} \neq \pi_n\}\right )\\
&= \lim_{n\to\infty} \mu_n\left (\{(\pi^{\theta_n})_{|[k]_{\star}} \in\point\}\cap\{\abs{\pi}^{\downarrow} \neq \mathbf{1}\}\right )\\
&= \lim_{n\to\infty} \overleftarrow{\mu}_n\left (\{\pi_{|[k]_{\star}} \in\point\}\cap\{\abs{\pi}^{\downarrow} \neq \mathbf{1}\}\right )\\
&= \lim_{n\to\infty} \int_{\masspartemb\setminus\{\mathbf{1}\}} \mu_n(\abs{\pi}^{\downarrow} \in\dd \mathbf{p}) \bar{\rho}_{\mathbf{p}}(\pi_{|[k]_{\star}} \in\point)\\
&= \int_{\masspartemb\setminus\{\mathbf{1}\}} \mu(\abs{\pi}^{\downarrow} \in\dd \mathbf{p}) \;\bar{\rho}_{\mathbf{p}}(\pi_{|[k]_{\star}} \in\point).
\end{align*}
Note that the passage from the second to the third line follows from invariance of~$\mu$ under the permutation $ \sigma : \N \to\N $ defined by
\[
  \sigma : \begin{cases}
    i \in\{1,\ldots k\} &\mapsto i+n,\\
    i \in \{k+1, \ldots, k+n\} &\mapsto i-k,\\
    i \geq k+n+1 &\mapsto i.
  \end{cases}
\]
Since this is true for all $k$, we have
\[ \phi(\point) = \int_{\masspartemb\setminus\{\mathbf{1}\}} \mu(\abs{\pi}^{\downarrow} \in\dd \mathbf{p}) \;\bar{\rho}_{\mathbf{p}}(\point) = \bar{\rho}_{\nu}, \]
with $\nu(\point) = \mu(\{\abs{\pi}^{\downarrow} \in\point\}\cap\{\abs{\pi}^{\downarrow} \neq \mathbf{1}\})$.
Now notice that the paintbox construction of the probabilities $\bar{\rho}_\mathbf{p}$ implies that
\[ \bar{\rho}_\nu(\pi_{|[n]_{\star}} \neq \pi_n) = \int_{\masspartemb}\nu(\dd \mathbf{p}) \bigg (1-\sum_{l\geq 1}u_l^{n}\bigg ), \]
and that since $u_1\geq u_2 \geq \ldots...$ and $\sum_l u_l\leq 1$, we have for $n\geq 2$,
\[ 1-u_1\leq 1 - u_1\textstyle\sum_l u_l^{n-1} \leq 1-\textstyle\sum_l u_l^{n} \leq 1-u_1^n \leq n(1-u_1). \]
Integrating with respect to $\nu$, we find that clearly $\bar{\rho}_\nu$ satisfies~\eqref{eq:mu_sig_finite_inner} iff $\nu$ satisfies~\eqref{eq:nu_condition}.

We now write $\psi(\point) := \mu(\point\cap\{\abs{\pi}^{\downarrow}=\mathbf{1}\})$ so that $\mu = \phi + \psi = \bar{\rho}_\nu + \psi$.
Take a number $n\in\N$.
We know that $\overleftarrow{\psi}_n(\point) :=\psi(\{\pi^{\theta_n}\in\point\}\cap\{\pi_{|[n]_{\star}} \neq \pi_n\})$ is a finite exchangeable measure on $\partitemb{\infty,\star}$ such that $\abs{\pi}^{\downarrow}=\mathbf{1}$ $\overleftarrow{\psi}_n$-a.e.
Now recall that $\bar{\rho}_{\mathbf{1}}=\delta_{\pi_\infty}$.
A consequence of Lemma~\ref{lem:partitemb_asymptotic_frequencies} is that $\pi = \pi_\infty$ $\overleftarrow{\psi}_n$-a.e., which in turn implies that $\psi$-a.e. on the event $\{\pi_{|[n]_{\star}} \neq \pi_n\}$, we have $\pi^{\theta_n} = \pi_\infty$.
Since there is only a finite number of elements $\pi\in\partitemb{\infty,\star}$ such that $\pi^{\theta_n} = \pi_\infty$, we have
\[ \psi(\point\cap\{\pi_{|[n]_{\star}} \neq \pi_n\}) = \sum_i a_i \delta_{\widehat{\pi}_i}, \]
where the sum is finite, and for each $i$, we have $\widehat{\pi}_i^{\theta_n} = \pi_\infty$.
Now suppose we have $\psi(\{\widehat{\pi}\})>0$, for a $\widehat{\pi}\in\partitemb{\infty,\star}$ such that $\widehat{\pi}^{\theta_n} = \pi_\infty$.
Let $I(\widehat{\pi}) := \{\widehat{\pi}^{\sigma}, \sigma \text{ permutation}\}$.
By the exchangeability of $\psi$, we have necessarily $\psi(\{\pi\}) = \psi(\{\widehat{\pi}\}) > 0$ for any $\pi\in I(\widehat{\pi})$.
Since for any $m\in\N$ we have $\psi(\pi_{|[m]_{\star}}\neq\pi_m)<\infty$, we deduce
\begin{equation}\label{eq:nb_perm_finite}
\#\{\pi\in I(\widehat{\pi}), \; \pi_{|[m]_{\star}} \neq \pi_m\} \leq \psi(\pi_{|[m]_{\star}}\neq\pi_m) / \psi(\{\widehat{\pi}\}) < \infty.
\end{equation}
We claim that the elements $\widehat\pi=(\widehat{\zeta},\widehat{\xi})\in\partitemb{\infty, \star}$ satisfying $\widehat{\pi}^{\theta_n} = \pi_\infty$ and~\eqref{eq:nb_perm_finite} for any $m$ are such that
$\widehat{\zeta}$ and $\widehat{\xi}$ have no more than two blocks, and in that case one of the blocks is a singleton.
Indeed if $1\sim 2\nsim 3 \sim 4$ for $\widehat{\xi}$ or $\widehat{\zeta}$, then the permutations $\sigma_i = (2,i+2) (4,i+4)$, written as a composition of two transpositions, are such that for $i\neq j \geq n$ and $m\geq 3$, $\widehat{\pi}^{\sigma_i} \neq \widehat{\pi}^{\sigma_j}$ and $\widehat{\pi}^{\sigma_i}_{|[m]_{\star}} \neq \pi_m$.
So having two blocks with two or more integers contradicts~\eqref{eq:nb_perm_finite}.
One can check in the same way that the situation $1\nsim 2 \nsim 3$ is also contradictory.

Putting everything together, we necessarily have
\begin{itemize}
\item either $\widehat{\pi} = (\{\{\star\}, \{i\}, \N\setminus\{i\}\}, \mathbf{1}_{[\infty]_{\star}})$ for an $i\in\N$,
\item or $\widehat{\pi} = (\{\{\star\}, \{i\}, \N\setminus\{i\}\}, \{\{i\}, [\infty]_{\star}\setminus\{i\}\})$ for an $i\in\N$.
\end{itemize}
We conclude using the exchangeability of $\psi$ that there exists two real numbers $c_1,c_2 \geq 0$ such that $\psi = c_1\mathfrak{e}^1 + c_2\mathfrak{e}^2$, enabling us to write
\[ \mu = \phi+\psi = \bar{\rho}_\nu + c_1\mathfrak{e}^1 + c_2\mathfrak{e}^2, \]
which concludes the proof.
\end{proof}

Applying this result to $ \widetilde{\mu}_{\mathrm{in}} $ implies the existence of $c_{\mathrm{in},1},c_{\mathrm{in},2}\geq 0$ and $\nu_{\mathrm{in}}$ a measure on $\masspartemb$ satisfying~\eqref{eq:nu_condition} such that 
\[ \widetilde{\mu}_{\mathrm{in}} = c_{\mathrm{in},1}\mathfrak{e}^1 + c_{\mathrm{in},2}\mathfrak{e}^2+\bar{\rho}_{\nu_{\mathrm{in}}}. \]
This allows us to conclude the proof of Theorem~\ref{thm:inner_charac} because with our definitions in Section~\ref{sec:examples}, we translate this equality into
\[ \mu_{\mathrm{in}} =c_{\mathrm{in},1}\mathfrak{e}^{\mathrm{in},1}+c_{\mathrm{in},2}\mathfrak{e}^{\mathrm{in},2} + \widetilde{\rho}_{\nu_{\mathrm{in}}}. \]
Combining this with Lemma~\ref{lem:intermediate}, we conclude
\[ \mu =c_{\mathrm{out}}\mathfrak{e}^{\mathrm{out}}+c_{\mathrm{in},1}\mathfrak{e}^{\mathrm{in},1}+c_{\mathrm{in},2}\mathfrak{e}^{\mathrm{in},2} +\widehat{\rho}_{\nu_{\mathrm{out}}}+ \widetilde{\rho}_{\nu_{\mathrm{in}}}. \]

\section{Application to binary branching}\label{sec:binary}

Consider a simple nested fragmentation process $(\Pi(t),\,t\geq 0) = (\zeta(t),\xi(t),\,t\geq 0)$ with only binary branching.
The representation given by Theorem~\ref{thm:inner_charac} then becomes quite simpler, because the dislocation measures $\nu_{\mathrm{out}}$ and $\nu_{\mathrm{in}}$ necessarily satisfy
\[ s_1 = 1-s_2 \qquad \nu_{\mathrm{out}}\text{-a.e.} \] 
and
\[ \begin{cases}
&u_1=1-u_2\\
\text{or }&s_{1,1}=1-s_{1,2}\\
\text{or }&u_1=1-s_{1,1}
\end{cases}
\qquad \nu_{\mathrm{in}}\text{-a.e.}, \]
i.e.\ their support is the set of mass partitions with only two nonzero terms, and no dust.
See Figure \ref{fig:binarytree} for an example of a nested discrete tree illustrating the three possible dislocation events corresponding to $\nu_{\mathrm{in}}$.
\begin{figure}[ht]
  \centering
  \includegraphics[width=.8\linewidth]{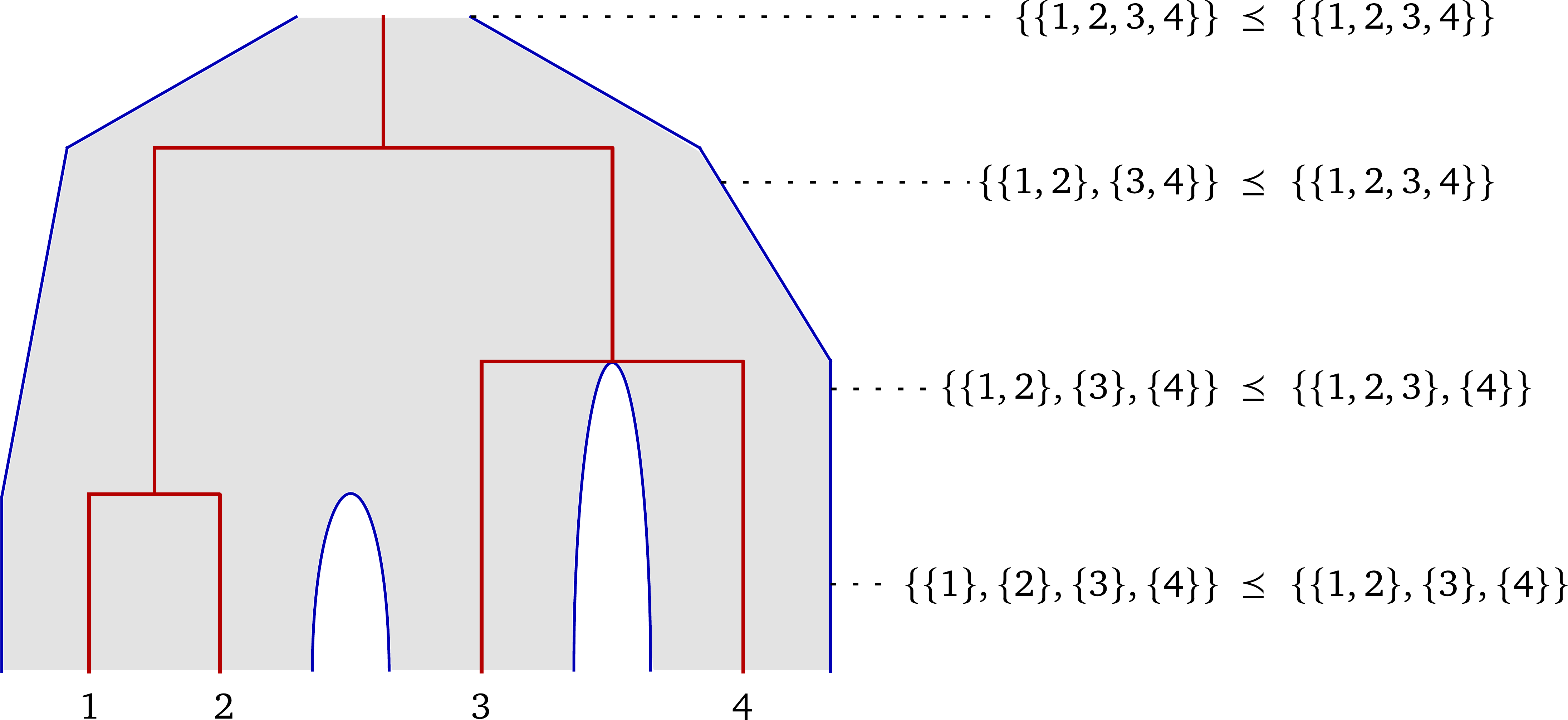}
  \captionsetup{justification=raggedright}
  \caption{Binary nested tree exhibiting the three different inner dislocation events.
    Time flows from top to bottom, and the right-hand side of the picture shows the sequence of nested partitions picked at chosen times between events, in the form $\pi = (\zeta \preceq \xi)$.
    The first event corresponds to the case $u_1=1-u_2$, where the inner block $\{1,2,3,4\}$ splits into two blocks $\{1,2\}$ and $\{3,4\}$ and the outer block remains unchanged.
    The second dislocation is of the type $u_1=1-s_{1,1}$, that is the block $\{3,4\}$ splits in two distinct blocks, one of which (the singleton $\{3\}$) stays in the ``mother'' outer block.
    The other new inner block $\{4\}$ forms a new outer block identical to itself.
    The last and third dislocation is of the type $s_{1,1}=1-s_{1,2}$, meaning that $\{1,2\}$ splits into $\{1\}$ and $\{2\}$, these two blocks together forming a new outer block, distinct from the mother block -- i.e.\ the one containing $\{3\}$.
  } \label{fig:binarytree}
  \captionsetup{justification=centering}
\end{figure}

Therefore, we can decompose $\nu_{\mathrm{out}}$ and $\nu_{\mathrm{in}}$ into four measures on $[0,1]$ defined by
\begin{gather*}
\bar\nu_{\mathrm{out}}(\point) := \nu_{\mathrm{out}}(s_1\in\point)+\nu_{\mathrm{out}}(1-s_1\in\point)\\
\bar\nu_{\mathrm{in},1}(\point) := \1\{u_1=1-u_2\}(\nu_{\mathrm{in}}(u_1\in\point)+\nu_{\mathrm{in}}(1-u_1\in\point))\\
\bar\nu_{\mathrm{in},2}(\point) := \1\{s_{1,1}=1-s_{1,2}\}(\nu_{\mathrm{in}}(s_{1,1}\in\point)+\nu_{\mathrm{in}}(1-s_{1,1}\in\point))\\
\bar\nu_{\mathrm{in},3}(\point) := \1\{u_1=1-s_{1,1}\}\nu_{\mathrm{in}}(u_1\in\point).
\end{gather*}
Thus defined, and because of the $\sigma$-finiteness conditions~\eqref{eq:nu_condition_simple} and~\eqref{eq:nu_condition}, those measures satisfy the following
\begin{gather}
\bar\nu_{\mathrm{out}},\bar\nu_{\mathrm{in},1}\text{ and }\bar\nu_{\mathrm{in},2}\text{ are }(x\mapsto 1-x)\text{-invariant}\\
\int_{[0,1]}\nu(\dd x) x(1-x) <\infty,\text{ for }\nu \in \{\bar\nu_{\mathrm{out}},\bar\nu_{\mathrm{in},1}\}\\
\bar\nu_{\mathrm{in},2}([0,1]) <\infty\\
\int_{[0,1]}\bar\nu_{\mathrm{in},3}(\dd x) (1-x) <\infty.
\end{gather}

\vspace{3cm} 
For the sake of completeness, let us use those measures to express the transition rates $q^{n}_{\pi,\pi'}$ of the Markov chain $\Pi^{n}:=(\Pi(t)\restr{n})$ from one nested partition $\pi=(\zeta,\xi)\in\partitemb{n}$ to another $\pi'=(\zeta',\xi')\in\partitemb{n}\setminus\{\pi\}$ in the following way:
\begin{itemize}
\item If $\pi'$ cannot be obtained from a binary fragmentation of $\pi$, then $q^{n}_{\pi,\pi'}=0$.
\item If $\pi'$ can be obtained from a binary fragmentation of $\pi$, with $B\in\zeta$ and $C\in\xi$ two blocks of $\pi$ participating in the fragmentation, but such that $B\nsubset C$, then $q^{n}_{\pi,\pi'}=0$.
\item Otherwise, let us write $B\subset C$, with $B\in\zeta$ and $C\in\xi$ for (the) two blocks of $\pi$ participating in the fragmentation, and $B_1,B_2\in\zeta'$, $C_1,C_2\in\xi'$ the resulting blocks, chosen in a way that $B_1\subset C_1$.
Note that $B$ or $C$ might not fragment, in which case we let $B_2$ or $C_2$ be the empty set $\emptyset$.
Now define $X_1:=\#B_1$ and $X_2:=\#B_2$ the cardinal of the resulting blocks of $\zeta'$, and $X:=\min(X_1,X_2)$.
Also, we define $Y_1:=\#\zeta'_{|C_1}$ the number of inner blocks in $C_1$ in the resulting partition $\pi'$, and similarly $Y_2:=\#\zeta'_{|C_2}$, and finally $Y:=\min(Y_1,Y_2)$.
\end{itemize}
With those definitions, the transition rates for the Markov chain $\Pi^{n}$ can be written
\begin{equation}\label{eq:transition_rates}
\begin{aligned}
q^{n}_{\pi,\pi'} = \;&c_{\mathrm{out}}\1\{\zeta'=\zeta,Y=1\} +c_{\mathrm{in},1}\1\{\xi'=\xi,X=1\} \\
&+c_{\mathrm{in},2}\1\{X_1=Y_1=1\text{ or }X_2=Y_2=1\}\\
&+ \1\{\zeta'=\zeta\}\int_{[0,1]}\bar\nu_{\mathrm{out}}(dx) x^{Y_1}(1-x)^{Y_2}\\
&+ \1\{\xi'=\xi\}\int_{[0,1]}\bar\nu_{\mathrm{in},1}(dx) x^{X_1}(1-x)^{X_2}\\
& +\1\{B_1\cup B_2=C_1\}\int_{[0,1]}\bar\nu_{\mathrm{in},2}(dx) x^{X_1}(1-x)^{X_2}\\
&+\1\{\zeta'=\zeta\text{ or }B_2\nsubset C_1\}\int_{[0,1]}\bar\nu_{\mathrm{in},3}(dx) \big(x^{X_1}(1-x)^{X_2}\1\{B_2=C_2\} \\ &\hspace{16em}+x^{X_2}(1-x)^{X_1}\1\{B_1=C_1\}\big).
\end{aligned}
\end{equation}
Note that several indicator functions in the last display may be one for the same pair $(\pi,\pi')$.
This explicit formula allows for computer simulations of binary simple nested fragmentations, although to that aim it might be simpler to adapt the Poissonian construction (Section~\ref{sec:poisson_construct}) and use nested partitions on arrays $[n]^2$.
Also, one could exactly compute the probability of a given nested tree under different nested fragmentation models, which would be a first step towards statistical inference.



\paragraph{Acknowledgements.}
I thank the \emph{Center for Interdisciplinary Research in Biology} (Collège de France) for funding, and I am grateful to my supervisor Amaury Lambert for his careful reading and many helpful comments on this project.

\phantomsection
\addcontentsline{toc}{section}{References}

\emergencystretch=.5em


\begin{thebibliography}{23}
  \providecommand{\natexlab}[1]{#1}
  \providecommand{\url}[1]{\texttt{#1}}
  \expandafter\ifx\csname urlstyle\endcsname\relax
  \providecommand{\doi}[1]{doi: #1}\else
  \providecommand{\doi}{doi: \begingroup \urlstyle{rm}\Url}\fi
  
  \bibitem[Aldous(1996)]{aldousProbability1996}
  D.~Aldous.
  \newblock Probability distributions on cladograms.
  \newblock In D.~Aldous and R.~Pemantle, editors, \emph{Random Discrete
    Structures}, pages 1--18. {Springer New York}, 1996.
  \newblock \doi{10.1007/978-1-4612-0719-1_1}.
  
  \bibitem[Berestycki(2009)]{Ber09}
  N.~Berestycki.
  \newblock Recent progress in coalescent theory.
  \newblock \emph{Ensaios Matem{\'a}ticos}, 16\penalty0 (1):\penalty0 1--193,
  2009.
  
  \bibitem[Bertoin(2006)]{bertoinRandom2006}
  J.~Bertoin.
  \newblock \emph{Random {{Fragmentation}} and {{Coagulation Processes}}}.
  \newblock {Cambridge University Press}, 2006.
  \newblock \doi{10.1017/CBO9780511617768}.
  
  \bibitem[Blancas et~al.(2018{\natexlab{a}})Blancas, Duchamps, Lambert, and
  Siri-J{\'e}gousse]{snec}
  A.~Blancas, J.-J. Duchamps, A.~Lambert, and A.~Siri-J{\'e}gousse.
  \newblock Trees within trees: Simple nested coalescents.
  \newblock \emph{arXiv:1803.02133}, Mar. 2018{\natexlab{a}}.
  
  \bibitem[Blancas et~al.(2018{\natexlab{b}})Blancas, Rogers, Schweinsberg, and
  {Siri-J{\'e}gousse}]{BRSS18}
  A.~Blancas, T.~Rogers, J.~Schweinsberg, and A.~{Siri-J{\'e}gousse}.
  \newblock The nested {{Kingman}} coalescent: Speed of coming down from
  infinity.
  \newblock \emph{arXiv:1803.08973}, Mar. 2018{\natexlab{b}}.
  
  \bibitem[Chen et~al.(2009)Chen, Ford, and Winkel]{chenNew2009}
  B.~Chen, D.~Ford, and M.~Winkel.
  \newblock A new family of {{Markov}} branching trees: The alpha-gamma model.
  \newblock \emph{Electronic Journal of Probability}, 14:\penalty0 400--430,
  2009.
  \newblock \doi{10.1214/EJP.v14-616}.
  
  \bibitem[Crane(2017)]{craneGeneralized2017}
  H.~Crane.
  \newblock Generalized {{Markov}} branching trees.
  \newblock \emph{Advances in Applied Probability}, 49\penalty0 (01):\penalty0
  108--133, Mar. 2017.
  \newblock \doi{10.1017/apr.2016.81}.
  
  \bibitem[Crane and Towsner(2016)]{craneStructure2016}
  H.~Crane and H.~Towsner.
  \newblock The structure of combinatorial {{Markov}} processes.
  \newblock \emph{arXiv:1603.05954}, Mar. 2016.
  
  \bibitem[Doyle(1997)]{doyleTrees1997}
  J.~J. Doyle.
  \newblock Trees within trees: Genes and species, molecules and morphology.
  \newblock \emph{Systematic Biology}, 46\penalty0 (3):\penalty0 537--553, Sept.
  1997.
  \newblock \doi{10.1093/sysbio/46.3.537}.
  
  \bibitem[Etheridge(2011)]{Eth11}
  A.~Etheridge.
  \newblock \emph{Some Mathematical Models from Population Genetics:
    {{{\'E}cole}} d'ete de Probabilit{\'e}s de {{Saint}}-{{Flour XXXIX}}-2009}.
  \newblock Number 2012 in Lecture notes in mathematics. {Springer}, 2011.
  
  \bibitem[Ford(2006)]{fordProbabilities2006}
  D.~J. Ford.
  \newblock \emph{Probabilities on Cladograms: Introduction to the Alpha Model.}
  \newblock PhD thesis, Stanford University, 2006.
  \newblock URL \url{https://arxiv.org/abs/math/0511246}.
  
  \bibitem[Foucart(2011)]{Fou11}
  C.~Foucart.
  \newblock Distinguished exchangeable coalescents and generalized
  {{Fleming}}-{{Viot}} processes with immigration.
  \newblock \emph{Advances in Applied Probability}, 43\penalty0 (02):\penalty0
  348--374, June 2011.
  \newblock \doi{10.1239/aap/1308662483}.
  
  \bibitem[Haas et~al.(2008)Haas, Miermont, Pitman, and
  Winkel]{haasContinuum2008}
  B.~Haas, G.~Miermont, J.~Pitman, and M.~Winkel.
  \newblock Continuum tree asymptotics of discrete fragmentations and
  applications to phylogenetic models.
  \newblock \emph{The Annals of Probability}, 36\penalty0 (5):\penalty0
  1790--1837, Sept. 2008.
  \newblock \doi{10.1214/07-AOP377}.
  
  \bibitem[Kingman(1982)]{kingmanCoalescent1982}
  J.~Kingman.
  \newblock The coalescent.
  \newblock \emph{Stochastic processes and their applications}, 13\penalty0
  (3):\penalty0 235--248, 1982.
  \newblock \doi{10.1016/0304-4149(82)90011-4}.
  
  \bibitem[Lambert(2008)]{Lam08}
  A.~Lambert.
  \newblock Population {{Dynamics}} and {{Random Genealogies}}.
  \newblock \emph{Stochastic Models}, 24\penalty0 (sup1):\penalty0 45--163, 2008.
  \newblock \doi{10.1080/15326340802437728}.
  
  \bibitem[Lambert(2017)]{lambertProbabilistic2017}
  A.~Lambert.
  \newblock Probabilistic models for the (sub)tree(s) of life.
  \newblock \emph{Brazilian Journal of Probability and Statistics}, 31\penalty0
  (3):\penalty0 415--475, Aug. 2017.
  \newblock \doi{10.1214/16-BJPS320}.
  
  \bibitem[Lambert and Schertzer()]{LS}
  A.~Lambert and E.~Schertzer.
  \newblock Coagulation-transport equations and nested coalescents.
  \newblock \emph{In preparation}.
  
  \bibitem[Maddison(1997)]{maddisonGene1997}
  W.~P. Maddison.
  \newblock Gene trees in species trees.
  \newblock \emph{Systematic Biology}, 46\penalty0 (3):\penalty0 523--536, Sept.
  1997.
  \newblock \doi{10.1093/sysbio/46.3.523}.
  
  \bibitem[Page and Charleston(1997)]{pageGene1997}
  R.~D. Page and M.~A. Charleston.
  \newblock From gene to organismal phylogeny: Reconciled trees and the gene
  tree/species tree problem.
  \newblock \emph{Molecular Phylogenetics and Evolution}, 7\penalty0
  (2):\penalty0 231--240, Apr. 1997.
  \newblock \doi{10.1006/mpev.1996.0390}.
  
  \bibitem[Page and Charleston(1998)]{pageTrees1998}
  R.~D. Page and M.~A. Charleston.
  \newblock Trees within trees: Phylogeny and historical associations.
  \newblock \emph{Trends in Ecology \& Evolution}, 13\penalty0 (9):\penalty0
  356--359, Sept. 1998.
  \newblock \doi{10.1016/S0169-5347(98)01438-4}.
  
  \bibitem[Pitman(1999)]{pitmanCoalescents1999}
  J.~Pitman.
  \newblock Coalescents with multiple collisions.
  \newblock \emph{The Annals of Probability}, 27\penalty0 (4):\penalty0
  1870--1902, Oct. 1999.
  \newblock \doi{10.1214/aop/1022874819}.
  
  \bibitem[Sagitov(1999)]{sagitovGeneral1999}
  S.~Sagitov.
  \newblock The general coalescent with asynchronous mergers of ancestral lines.
  \newblock \emph{Journal of Applied Probability}, 36\penalty0 (4):\penalty0
  1116--1125, Dec. 1999.
  \newblock \doi{10.1017/S0021900200017903}.
  
  \bibitem[Semple and Steel(2003)]{semplePhylogenetics2003}
  C.~Semple and M.~Steel.
  \newblock \emph{Phylogenetics}.
  \newblock Number~24 in Oxford lecture series in mathematics and its
  applications. {Oxford University Press}, 2003.
  
\end{thebibliography}

\end{document}